\documentclass[reqno,11pt]{amsart}
\usepackage{amsfonts}
\usepackage{amsmath,amssymb,amsthm,amsfonts}
\usepackage{mathrsfs}
\usepackage{bbm}
\usepackage{hyperref}

\usepackage{color}

\newtheorem{lemma}{Lemma}[section]
\newtheorem{theorem}{Theorem}[section]

\newtheorem{proposition}{Proposition}[section]

\numberwithin{equation}{section}

\newcommand{\R}{\mathbb{R}}

\newcommand{\Z}{\mathbb{Z}}

\renewcommand{\S}{\mathbb{S}}
\newcommand{\T}{\mathbb{T}}

\newcommand{\FM}{\mathbf{M}}
\newcommand{\FN}{\mathbf{N}}
\newcommand{\FP}{\mathbf{P}}

\newcommand{\FI}{\mathbf{I}}
\newcommand{\FX}{\mathbf{X}}

\newcommand{\FR}{\mathbf{R}}

\newcommand{\CA}{\mathcal{A}}

\newcommand{\CE}{\mathcal{E}}

\newcommand{\na}{\nabla}

\newcommand{\al}{\alpha}
\newcommand{\be}{\beta}
\newcommand{\ga}{\gamma}
\newcommand{\om}{\omega}

\newcommand{\la}{\lambda}
\newcommand{\de}{\delta}

\newcommand{\pa}{\partial}
\newcommand{\ka}{\kappa}
\newcommand{\eps}{\epsilon}

\newcommand{\ta}{\theta}

\newcommand{\vps}{\varepsilon}

\newcommand{\Ga}{\Gamma}

\newcommand{\eqdef}{\overset{\mbox{\tiny{def}}}{=}}

\begin{document}
\title[Hydrodynamic approximation with viscous heating to the Boltzmann equation]
{Incompressible hydrodynamic approximation with viscous heating to the Boltzmann equation}

\author[Y. Guo]{Yan Guo}
\address[YG]{Division of Applied Mathematics, Brown University, Providence 02912, USA}
\email{${\rm Yan}_-{\rm Guo}$@brown.edu}

\author[S.-Q. Liu]{Shuangqian Liu}
\address[SQL]{Department of Mathematics, Jinan University, Guangzhou 510632, P.R.~China}
\email{tsqliu@jnu.edu.cn}


\begin{abstract}
The incompressible Navier-Stokes-Fourier system with viscous heating was first derived from the Boltzmann equation in the form of the diffusive scaling by
Bardos-Levermore-Ukai-Yang (2008). The purpose of this paper is to justify such an incompressible hydrodynamic approximation to the Boltzmann equation in $L^2\cap L^\infty$ setting in a periodic box. Based on an odd-even expansion of the solution with respect to the microscopic
velocity, the diffusive coefficients are determined by the incompressible Navier-Stokes-Fourier system with viscous heating
and the super Burnett functions. More importantly, the remainder of the expansion is proven to decay exponentially in time via an $L^2-L^\infty$ approach on the condition that the initial data satisfies the mass, momentum and energy conversation laws.
\end{abstract}



\maketitle

\thispagestyle{empty}
\tableofcontents

\section{Introduction}
\subsection{The problem}
This paper is concerned with the connection between the incompressible fluid dynamical equations
with viscous heating and the Boltzmann equation in a periodic box.
In the diffusive regime, the time evolution of the dilute gas is governed by the following \textit{%
rescaled} Boltzmann equation:
\begin{equation}
\eps\pa_t F+v\cdot\nabla_x F=\frac{1}{\eps}Q(F,F),\ x\in\T^3,\ v\in\R^3,
\label{BE}
\end{equation}
with initial data
\begin{equation}\label{ID}
F(0,x,v) = F_0(x,v),\ x\in\T^3,\ v\in\R^3.
\end{equation}
Here, $F(t,x,v)\geq0$ is the distribution function of particles at time $t\in \mathbb{R}_{+}$, position $x\in[-\pi,\pi]^3=\T^3$ and velocity $v\in \mathbb{
R}^{3}$, and $\eps>0$ is the Knudsen number which is proportional to the
mean free path.

$Q(\cdot,\cdot)$ in \eqref{BE} is the Boltzmann collision operator, which for the hard sphere model takes the following non-symmetric form:
$$
Q(F,H)=\int_{\R^3\times \S_+^2}
\Big(F(v
_\ast')H(v')-F(v_\ast)H(v)\Big)|(v-v_\ast)\cdot\om|\,dv_\ast d\om,
$$
where $\S_+^2=\{\om\in\S^2: (v-v_\ast)\cdot\om\geq0\}$
and $(v,v_{\ast})$, and $(v',v_{\ast}')$, denote velocities of two particles before and after an elastic collision respectively, satisfying
\begin{equation*}
v'=v-[(v-v_\ast)\cdot\om]\om,\quad
v_\ast'=v_\ast+[(v-v_\ast)\cdot\om]\om.
\end{equation*}
Recently, there have been great interest \cite{Bardos-Golse-Levermore-1993,EGKM-15,Golse-2005,Gosle-Saint-Raymond-2004,Guo-2006,Levermore-Masmoudi-2010,Masmoudi-Saint-Raymond-2003} in studying the rescaled Boltzmann
equation \eqref{BE}, which naturally leads to the incompressible Navier-Stokes-Fourier (denoted by INSF in the sequel)
equations in the dimensionless form according to the Hilbert expansion. Among others, Bardos-Levermore-Uaki-Yang \cite{Bardos-Levermore-Ukai-Yang-2008} developed a so-called odd-even decomposition to derive a new incompressible hydrodynamic system which differs from the classical INSF
equations in that they include the viscous heating term 
and driving terms involving the limiting pressure fluctuation. The aim of the present paper is to employ the $L^{2}-L^{\infty }$ framework
developed in \cite{Guo-2010} to justify the validity of such an INSF
equations approximation with viscous heating
to the Boltzmann equation in a periodic box.

\subsection{Odd-even expansion with remainder}
Let $\mu$ be the global Maxwellian defined as
\begin{equation*}  \label{g_Maxwellian}
{\mu(v)=M_{[1,0,1]} = \frac{1}{(2\pi)^{3/2}} e^{- \frac{|v|^{2}}{2}} . }
\end{equation*}
The odd-even expansion \cite{Bardos-Levermore-Ukai-Yang-2008} suggests that the solution of \eqref{BE} can be written as
\begin{equation}\label{oe.ep}
F=\mu+\eps\sqrt{\mu}\left\{f_1+\eps f_2+\eps^2f_3+\eps^3f_4+\eps^4 f_5+\eps^5 f_6+\eps^{4-\beta}R\right\},\ 0<\beta<1/2,
\end{equation}
where
\begin{equation}\label{fioe}
f_1, f_3\  \text{and}\ f_5\ \text{are odd in}\ v,\ \text{while}\
f_2, f_4\ \text{and}\ f_6\ \text{are even in}\ v.
\end{equation}
Plugging \eqref{oe.ep} into \eqref{BE} and comparing the coefficients on both side of the resulting equation, we obtain for $\eps^0$, $\eps^1$, $\eps^2$, $\eps^3$, $\eps^4$ and $\eps^5$
\begin{equation}\label{f1e}
Lf_1=0,
\end{equation}
\begin{equation}\label{f2e}
v\cdot \na_xf_1+Lf_2=\Gamma(f_1,f_1),
\end{equation}
\begin{equation}\label{f3e}
\pa_tf_1+v\cdot\na_xf_2+Lf_3=\Gamma(f_1,f_2)+\Gamma(f_2,f_1),
\end{equation}
\begin{equation}\label{f4e}
\pa_tf_2+v\cdot\na_xf_3+Lf_4=\Gamma(f_2,f_2)+\Gamma(f_1,f_3)+\Gamma(f_3,f_1),
\end{equation}
\begin{equation}\label{f5e}
\pa_tf_3+v\cdot\na_xf_4+Lf_5=\Gamma(f_2,f_3)+\Gamma(f_3,f_2)+\Gamma(f_4,f_1)+\Gamma(f_1,f_4),
\end{equation}
\begin{equation}\label{f6e}
\pa_tf_4+v\cdot\na_xf_5+Lf_6=\Gamma(f_3,f_3)+\Gamma(f_2,f_4)+\Gamma(f_4,f_2)+\Gamma(f_1,f_5)+\Gamma(f_5,f_1),
\end{equation}
and the equation for the remainder $R$
\begin{equation}\label{R}
\begin{split}
\eps\pa_tR&+v\cdot\na_xR+\frac{1}{\eps}LR\\
=&\left\{\Gamma(f_1,R)+\Gamma(R,f_1)\right\}
+\eps\left\{\Gamma(f_2,R)+\Gamma(R,f_2)\right\}\\
&+\eps^2\left\{\Gamma(f_3,R)+\Gamma(R,f_3)\right\}+\eps^3\left\{\Gamma(f_4,R)+\Gamma(R,f_4)\right\}
\\
&+\eps^4\left\{\Gamma(f_5,R)+\Gamma(R,f_5)\right\}
+\eps^5\left\{\Gamma(f_6,R)+\Gamma(R,f_6)\right\}
\\&+\eps^{4-\be}\Gamma(R,R)
-\eps^{1+\beta}\left\{\pa_tf_5+v\cdot\na_xf_6\right\}-\eps^{2+\beta} \pa_t f_6,
\end{split}
\end{equation}
with
\begin{equation}\label{RID}
R(0,x,v)=R_0(x,v).
\end{equation}
Here, the linear collision operator $L$ and nonlinear collision operator $\Gamma$ are defined as
\begin{equation*}
Lg=-\frac{1}{\sqrt{\mu}}\big\{Q(\mu,\sqrt{\mu}g)+Q(\sqrt{\mu}g,\mu)\big\},
\end{equation*}
and
\begin{equation*}
\Gamma(g,h)=\frac{1}{\sqrt{\mu}}Q(\sqrt{\mu}g,\sqrt{\mu}h),
\end{equation*}
respectively.  The null space of $L$ denoted by
$\mathscr{N}(L)$ is generated by
$\left[\sqrt{\mu},v\sqrt{\mu},v^{2}\sqrt{\mu}\right]$, thus for any
function $g(t,x,v)$, we can decompose it as follows
$$
g={\bf P}g+\{{\bf I- P}\}g,
$$
where ${\bf P}g$ is the $L^{2}_{v}-$projection of $g$ on the null
space for $L$ for given $(t,x)$ and we can further denote ${\bf P}g$
by
\begin{equation*}
{\bf P}g=\left\{\rho_{g}(t,x)+v\cdot
u_{g}(t,x)+\frac{|v|^{2}-3}{2}\theta_{g}(t,x)\right\}\sqrt{\mu}.
\end{equation*}
Here $\rho_{g}(t,x),u_{g}(t,x),$ and $\theta_{g}(t,x)$ also represent
the density, velocity, and temperature fluctuation physically
respectively. It is traditional to call $\FP g$ the macroscopic part and $\{\FI-\FP\}g$ the microscopic part. 
In addition, the linearized Boltzmann collision operator $L$
satisfies 
$$
Lg=\nu(v)g-Kg,
$$
where $\nu(v)$ is called the collision frequency which is given by
\begin{equation*}\label{1.21}
\nu(v)=\int_{{\R}^{3}\times \S^2}|(v-v_\ast)\cdot\om|\mu(u)B(\theta)dv_\ast d\omega\thicksim \langle v\rangle=\sqrt{1+|v|^2},
\end{equation*}
and operator $K=K_{2}-K_{1}$ is defined
as in the following 
\begin{eqnarray}\label{defK}
\left\{\begin{array}{rll}
\begin{split}
[K_{1}g](v)=&\int_{{\R}^{3}\times \S^2}|(v-v_\ast)\cdot\om|\mu^{\frac{1}{2}}(v_\ast)\mu^{\frac{1}{2}}(v)g(v_\ast)dv_\ast d\omega,\\
[K_{2}g](v)=&\int_{{\R}^{3}\times \S^2}|(v-v_\ast)\cdot\om|\mu^{\frac{1}{2}}(v_\ast)\left\{\mu^{\frac{1}{2}}(v_\ast')g(v')
+\mu^{\frac{1}{2}}(v')g(v_\ast')\right\}dv_\ast d\omega.
\end{split}
\end{array}\right.
\end{eqnarray}
It is well known that $L\geq
0$ and there exists $\de_0>0$ such that
\begin{equation*}\label{Lco}
\langle Lg,g\rangle \geq \de_0 \|\{{\bf I- P}\}g\|^2_{\nu}.
\end{equation*}
For later use, we also define the following Burnett functions $A(v)$ and $B(v)$
as
$$
A(v)=(A(v)_{ij})_{3\times3}=\left\{v\otimes
v-\frac{1}{3}|v|^{2}I\right\}\sqrt{\mu},\ \ B(v)=(B_j(v))_{3\times1}=\frac{|v|^2-5}{2}v\sqrt{\mu},
$$
where $I$ is the identity matrix.

From \eqref{fioe}, \eqref{f1e}, \eqref{f2e}, \eqref{f3e}, \eqref{f4e}, \eqref{f5e} and \eqref{f6e}, one can further define $f_1, f_2, f_3$, $f_4$, $f_5$, and $f_6$ as follows:
\begin{eqnarray}\label{f1-5}
\left\{\begin{array}{rlll}
\begin{split}
&f_1=u_1\cdot v\sqrt{\mu},\\[2mm]
&f_2=\left\{\rho_1+\frac{1}{6}(|u_1|^2+3\ta_1)(|v|^2-3)\right\}\sqrt{\mu}-L^{-1}[A(v)]:\na_xu_1+\frac{1}{2}A(v):u_1\otimes u_1,\\[2mm]
&f_3=u_2\cdot v\sqrt{\mu}+\rho_1u_1\cdot v\sqrt{\mu}+L^{-1}\left\{-v\cdot\na_xf_2+\Gamma(f_1,f_2)+\Gamma(f_2,f_1)\right\},\\[2mm]
&f_4=
\left\{\rho_2+\frac{1}{6}(2u_1\cdot u_2+3\ta_2)(|v|^2-3)+\frac{1}{6}\rho_1u_1^2(|v|^2-3)\right\}\sqrt{\mu}+\rho_1\ta_1\frac{|v|^2-3}{2}\sqrt{\mu}
\\[2mm]
&\qquad+\ta_1^2\frac{|v|^2}{2}\sqrt{\mu}+
L^{-1}\left\{-\pa_tf_2-v\cdot\na_xf_3
+\Gamma(f_2,f_2)+\Gamma(f_1,f_3)+\Gamma(f_3,f_1)\right\},\\[2mm]
&f_5=u_3\cdot v\sqrt{\mu}+\rho_1 u_2\cdot v\sqrt{\mu}+\rho_2 u_1\cdot v\sqrt{\mu}\\[2mm]&\qquad+L^{-1}\left\{-\pa_tf_3-v\cdot\na_xf_4
+\Gamma(f_2,f_3)+\Gamma(f_3,f_2)+\Gamma(f_1,f_4)+\Gamma(f_4,f_1)\right\},\\[2mm]
&f_6=
L^{-1}\left\{-\pa_tf_4-v\cdot\na_xf_5
+\sum\limits_{i+j=6}\{\Gamma(f_i,f_j)+\Gamma(f_j,f_i)\}\right\}.
\end{split}
\end{array}\right.
\end{eqnarray}
Here, we can take $\FP f_6=0$,
since the expansion is truncated.
It is worth stressing that
the macroscopic parts of $f_1$, $f_2$, $f_3$, $f_4$ and $f_5$ stem from the following Taylor expansion of the local Maxwellian:
\begin{equation*}\label{Mexp}
\begin{split}M&_{[1+\eps^2\rho_1+\eps^4\rho_2,\eps u_1+\eps^3 u_2+\eps^5 u_3,
1+\eps^2\ta_1+\eps^4\ta_2]}\\=&\frac{1+\eps^2\rho_1+\eps^4\rho_2}{(2\pi(1+\eps^2\ta_1+\eps^4\ta_2))^{3/2}} e^{- \frac{|v-\eps u_1-\eps^3 u_2-\eps^5u_3|^{2}}{2(1+\eps^2\ta_1+\eps^4\ta_2)}}\\
=&M_{[1,0,1]}\bigg\{\eps u_1\cdot v+\eps^2\Big[\rho_1+\frac{1}{6}(|u_1|^2+3\ta_1)(|v|^2-3)+\frac{1}{2}A(v):u_1\otimes u_1\Big]
\\&+\eps^3\Big[u_2\cdot v+\rho_1u_1\cdot v+\ta_1u_1\cdot B(v)+\frac{1}{6}P_1(v):u_1\otimes u_1\otimes u_1\Big]
\\&+\eps^4\Big[\rho_2+\frac{1}{6}(2u_1\cdot u_2+3\rho_1\ta_1+3\ta_2)(|v|^2-3)+\ta_1^2\frac{(|v|^2-3)(|v|^2-5)}{4}
+A(v):u_1\otimes u_2
\\&\qquad\quad+\frac{1}{2}\rho_1(v\otimes v-I):u_1\otimes u_1+\frac{1}{2}\ta_1(\frac{|v|^2-7}{2}v\otimes v-I\frac{|v|^2-5}{2}):u_1\otimes u_1
\\&\qquad\quad+\frac{1}{24}P_2(v):u_1\otimes u_1\otimes u_1\otimes u_1\Big]
\\&+\eps^5\Big[u_3\cdot v+\rho_1 u_2\cdot v+\rho_2 u_1\cdot v+\ta_1u_2\cdot B(v)+\ta_2u_1\cdot B(v)+\rho_1\ta_1u_1\cdot B(v)
\\&\qquad\quad+\frac{1}{2}P_1(v):u_1\otimes u_1\otimes u_2+\frac{1}{6}\rho_1P_1(v):u_1\otimes u_1\otimes u_1
\\&\qquad\quad+\frac{1}{6}\ta_1\big(v\otimes v\otimes v\frac{|v|^2-9}{2}
-3I\otimes v\frac{|v|^2-3}{2}+6v\otimes I\big):u_1\otimes u_1\otimes u_1
\\&\qquad\quad+\frac{1}{120}P_3(v):u_1\otimes u_1\otimes u_1\otimes u_1\otimes u_1\Big]+O(\eps^6)\bigg\},
\end{split}
\end{equation*}
where $P_1\sqrt{\mu}$, $P_2\sqrt{\mu}$ and $P_3\sqrt{\mu}$ are the {\it super Burnett functions} given by
\begin{eqnarray*}\label{Pi}
\left\{
\begin{array}{rll}
\begin{split}
P_1&=v\otimes v\otimes v-3I\otimes v,\\[2mm]
P_2&=v\otimes v\otimes v\otimes v-6 I\otimes v\otimes v+3I,\\[2mm]
P_3&=v\otimes v\otimes v\otimes v\otimes v-10 I\otimes v\otimes v\otimes v+15I\otimes v.
\end{split}
\end{array}\right.
\end{eqnarray*}
Moreover, here and in the sequel, we define
${\bf M}:\FN=\sum\limits_{i=1}^m \sum\limits_{j=1}^n a_{ij}b_{ij}$
for two $m\times n$ matrices $\FM =(a_{ij})$ and $\FN=(b_{ij})$.
It is also straightforward to check that $P_1(v):u_1\otimes u_1\otimes u_1$, $P_2(v):u_1\otimes u_1\otimes u_1\otimes u_1$,
$P_3(v):u_1\otimes u_1\otimes u_1\otimes u_1\otimes u_1$, $(\frac{|v|^2-7}{2}v\otimes v-I\frac{|v|^2-5}{2})$ and $\big(v\otimes v\otimes v\frac{|v|^2-9}{2}
-3I\otimes v\frac{|v|^2-3}{2}+6v\otimes I\big):u_1\otimes u_1\otimes u_1$ all belong to the orthogonal complement of $\mathscr{N}(L)$.

In addition, \eqref{f2e}, \eqref{f3e} and \eqref{f4e} give rise to the following so-called incompressible Navier-Stokes-Fourier equations with viscous heating
\begin{eqnarray}\label{INSF}
\left\{\begin{array}{rlll}
\begin{split}
&\na_x\cdot u_1=0,\\
&\pa_tu_1+u_1\cdot\na_xu_1+\na_xp_1=\mu_\ast\Delta u_1,\ \ p_1=\rho_1+\theta_1,\\
&\pa_t\left(\frac{3}{2}\ta_1-\rho_1\right)+u_1\cdot\na_x\left(\frac{3}{2}\ta_1-\rho_1\right)=\ka_\ast\Delta\theta_1+\frac{1}{2}\mu_\ast
\left|\na_xu_1+(\na_xu_1)^T\right|^2,\\
&\rho_1(0,x)=\rho_{1,0}(x),\ \ u_1(0,x)=u_{1,0}(x),\ \ \ta_1(0,x)=\ta_{1,0}(x),
\end{split}\end{array}\right.
\end{eqnarray}
and \eqref{f4e}, \eqref{f5e} and \eqref{f6e} lead us to
\begin{eqnarray}\label{INSF2}
\left\{\begin{array}{rlll}
\begin{split}
&\pa_t \rho_1+\na_x\cdot u_2+\na_x\cdot(\rho_1u_1)=0,\\[2mm]
&\pa_tu_2+u_1\na_x\cdot u_2+\na_x u_1\cdot u_2+u_1\cdot \na_x u_2
+\na_x\left(\rho_2+\ta_2-\frac{1}{3}u_1\cdot u_2\right)
\\[2mm]&\quad=\mu_\ast\Delta u_2+\frac{\mu_\ast}{3}\na_x\na_x\cdot\{\FI-\FP_0\}u_2
+\na_x\cdot\langle \pa_t f_2, L^{-1}A(v)\rangle
\\[2mm]&\qquad-\na_x\cdot\langle\Gamma(f_2,f_2),L^{-1}A(v)\rangle
-\na_x\cdot\langle\Gamma(f_1,\{\FI-\FP\}f_3)+\Gamma(\{\FI-\FP\}f_3,f_1),L^{-1}A(v)\rangle
\\[2mm]&\qquad-\pa_t(\rho_1u_1)-2\na_x\cdot(\rho u_1\otimes u_1)-\na_{x}(\rho_1\ta_1+\frac{5}{2}\ta_1^2-\frac{1}{3}\rho_1|u_1|^2)
\\[2mm]&\qquad+\mu_\ast\Delta (\rho_1u_1)+\frac{\mu_\ast}{3}\na_x\na_x\cdot(\rho_1u_1),\\[2mm]
&\pa_t \left(\frac{3}{2}\ta_2-\rho_2\right)+\frac{5}{2}\na_x\cdot(u_1\ta_2)+
\frac{5}{6}\na_x\cdot(u_1u_1\cdot u_2)\\[2mm]&\quad=\ka_\ast\Delta\ta_2+\ka_\ast\Delta(\rho_1\ta_1+\ta_1^2)+\frac{2\ka_\ast}{3}\Delta(u_1\cdot u_2)
+\frac{\ka_\ast}{3}\Delta(\rho_1 u_1^2)-\frac{1}{2}\pa_t(2u_1\cdot u_2+\rho_1 u_1^2+3\rho_1\ta_1)\\[2mm]&
\qquad-\frac{5}{2}\na_x\cdot(u_1(\rho_1\ta_1+\ta_1^2))-\frac{5}{6}\na_x\cdot(u_1\rho u_1^2)
\\[2mm]&
\qquad-\na_x\cdot \langle L^{-1}\left\{-\pa_tf_3-v\cdot\na_x\{\FI-\FP\}f_4
+\Gamma(f_2,f_3)+\Gamma(f_3,f_2)\right\},B(v)\rangle
\\[2mm]&
\qquad-\na_x\cdot \langle L^{-1}\left\{\Gamma(f_1,\{\FI-\FP\}f_4)+\Gamma(\{\FI-\FP\}f_4,f_1)\right\},B(v)\rangle,\\[2mm]
&\pa_t\rho_2+\na_x\cdot u_3+\na_x\cdot(\rho_1u_2)+\na_x\cdot(\rho_2u_1)=0,\\[2mm]
&\rho_2(0,x)=\rho_{2,0}(x),\ u_2(0,x)=u_{2,0}(x),\ta_2(0,x)=\ta_{2,0}(x).
\end{split}\end{array}\right.
\end{eqnarray}
Here, ${\FP_0}$ is the divergence free operator on torus and defined as
\begin{eqnarray}\label{Divop}
\left\{
\begin{array}{rll}
\begin{split}
{\bf P}_{0}u&=\sum\limits_{m\in \Z^3}\left[p_{0}(m)\int_{\T^3}u(x)e^{-2\pi im\cdot x}dx\right]e^{2\pi im\cdot x},\\[2mm]
p_{0}(m)&=\left(\delta_{jk}-\frac{m_{j}m_{k}}{|m|^2}\right)_{3\times
3},\ \de_{jk}\ \text{is the Kronecker delta}.
\end{split}
\end{array}\right.
\end{eqnarray}
Moreover, $\mu_\ast=\frac{1}{10}\langle L^{-1}[A(v)],A(v)\rangle$ and $\ka_\ast=\frac{1}{3}\langle L^{-1}[B(v)],B(v)\rangle$ represent the viscosity and heat conductivity, respectively. It should be pointed out that
\begin{equation*}
\begin{split}
\frac{1}{2}\mu_\ast
\left|\na_xu_1+(\na_xu_1)^T\right|^2\eqdef&\frac{1}{2}\mu_\ast {\text trace}\left((\na_xu_1+(\na_xu_1)^T)^2\right)\\
\eqdef &\frac{1}{2}\mu_\ast(\na_xu_1+(\na_xu_1)^T):(\na_xu_1+(\na_xu_1)^T)\\
=&\mu_\ast\left(\sum\limits_{i,j}\pa_ju_1^i\pa_iu_1^j+\sum\limits_{i}(\pa_iu_1^i)^2\right)
\end{split}
\end{equation*}
is the viscous heating term, which does not appear in the classical INSF equations, cf. \cite{Bardos-Golse-Levermore-1993,Guo-2006}.

\subsection{Main results}
For $l\geq0$, denote $w_l=\langle v\rangle^l=(1+|v|^2)^{l/2}.$ We now state our main results as follows:

\begin{theorem}\label{mre}
Let $F_0(x,v)=\mu+\eps\sqrt{\mu}\left\{\sum\limits_{r}^6\eps^{r-1}f_r(0,x,v)+\eps^{4-\beta}R_0(x,v)\right\}\geq 0$ with $0<\be<1/2$.
Assume
\begin{description}
  \item[$(\CA_1)$] $f_{r}(0,x,v)$ $(r=1,2,\cdots,6)$ possess the zero-mean hydrodynamic fields:
$$
(f_{r}(0,x,v),[1,v,(v^2-3)]\sqrt{\mu})=0,
$$
namely,
\begin{eqnarray}\label{meanm}
\left\{\begin{array}{rll}
\begin{split}
& \int_{\T^3}\rho_{1,0}dx=\int_{\T^3}\rho_{2,0}dx=0,\\
&\int_{\T^3}(3\ta_{1,0}+|u_{1,0}|^2)dx=\int_{\T^3}(3\ta_{2,0}+2u_{1,0}\cdot u_{2,0}+\rho_{1,0}|u_{1,0}|^2+3\rho_{1,0}\ta_{1,0})dx=0,\\
&\int_{\T^3}u_{1,0}dx=\int_{\T^3}(u_{2,0}+\rho_{1,0}u_{1,0})dx=\int_{\T^3}(u_{3,0}+\rho_{1,0}u_{2,0}+\rho_{2,0}u_{1,0})dx=0,
\end{split}
\end{array}\right.
\end{eqnarray}
in particular, the velocity fields also satisfy
\begin{equation*}\label{urdf}
\FP_0 u_{r,0}=u_{r,0}\ \text{for}\ r=1,2,3,
\end{equation*}
and there exists a sufficiently small $\vps_0>0$  such that
\begin{equation*}\label{pbd}
\|u_{1,0}\|_{H^4}+\|\ta_{1,0}\|_{H^{4}}\leq \vps_0;
\end{equation*}
 \item[$(\CA_2)$] for $l>3/2$,
 $\eps^{3/2}\|w_lR_0\|_{\infty}+\|R_0\|_2$ is sufficiently small and
\begin{equation*}\label{con.R0}
(R_0(x,v),[1,v,v^2]\sqrt{\mu})=0.
\end{equation*}
\end{description}
Then
the Cauchy problem \eqref{BE} and \eqref{ID} admits a unique global solution
$$
F(t,x,v)=\mu+\eps\sqrt{\mu}\left\{\sum\limits_{r=1}^6\eps^{r-1}f_r+\eps^{4-\beta}R\right\}\geq0,
$$
with $f_r$ $(r=1,2,\cdots,6)$ satisfying \eqref{f1-5}, \eqref{INSF} and \eqref{INSF2} and $R$ satisfying
\eqref{R} and \eqref{RID}, respectively. Moreover,
there exists a constant $\la>0$ and a polynomial $P$ with $P(0)=0$ such that for any $t\geq0$ and $l>3/2$
\begin{equation*}
\begin{split}
\eps^{3/2}&\|w_lR(t)\|_{\infty}+\|R(t)\|_2\\
\leq& Ce^{-\la t}\left\{\eps^{3/2}\|w_lR_0\|_{\infty}+\|R_0\|_2+\eps^{\be}P\left(
\|[u_{1,0},\ta_{1,0}]\|_{H^{16}}+\|[u_{2,0},\ta_{2,0}]\|_{H^{14}}\right)\right\}.
\end{split}
\end{equation*}
\end{theorem}

A great amount of effort has been paid on the study of the hydrodynamic approximation (limits) to the Boltzmann equation,
since the pioneering work by Hilbert, who introduce his famous expansion in terms of Knudsen number $\eps$ in \cite{Hilbert} to explore
the connection between the fluid dynamics and the Boltzmann equation. Grad \cite{Grad-1965} and Nishida \cite{Nishida-1978} investigated the asymptotic equivalence of the Boltzmann equation and the compressible Euler equations for gas dynamics, while Caflisch \cite{Caflisch-1980} and Lachowiz \cite{Lachowicz-1987} also studied the same issue by different methods. Mathematical descriptions on the closeness of the Chapman-Enskog expansion \cite{Chapman-Cowling-1990} of the Boltzmann equation to the solutions of the compressible Navier-Stokes equations were obtained by Lachowiz \cite{Lachowicz-1992}, Kawashima-Matsumura-Nishida \cite{Kawashima-Matsumura-Nishida-1979} and Liu-Yang-Zhao \cite{Liu-Yang-Zhao-2014}.

In the context of diffusive scaling, the problem can be faced only in the low mach number regime, in this situation, the Boltzmann solution shall be close to the INSF system, see in particular, a formal derivation by Bardos-Golse-Levermore \cite{Bardos-Golse-Levermore-1991} and \cite{Bardos-Golse-Levermore-1993} for a general momentum argument of deriving global Levay solution of INSF from global renormalized solution \cite{DiPerna-Lions-1989}
of the Boltzmann equation with additional assumption which remained unverified. Later on, there are a huge number of papers concerning this topic, see \cite{Golse-2005,Golse-Saint-Raymond-2009,Jiang-Levermore-Masmoudi-2010,Jiang-Masmoudi-2015,Levermore-Masmoudi-2010,Lions-Masmoudi-2001,
Masmoudi-Saint-Raymond-2003,Saint-Raymond-2009}. We point out that some of assumptions in \cite{Bardos-Golse-Levermore-1993} have been removed in those works. A full proof for the INSF limits of the Boltzmann equation has been given by Golse-Saint-Raymond \cite{Gosle-Saint-Raymond-2004}. There also have been extensive investigation on the convergence of the smooth solutions of the INSF system to the Botlzamnn equation, see \cite{Bardos-Ukai,DeMasi-Esposito-Lebowitz-1989,Esposito-Pulvirenti-2004,Guo-2006,Liu-Zhao-2011,Ukai-Asano-1983}.

We also mention that when the solutions of the Boltzmann equation are a small perturbation of some nontrivial profiles, for instant, some basic wave patterns, stationary solutions, time-periodic solutions, etc., the time-asymptotic equivalence of the Boltzmann equation and the compressible Navier-Stokes equation are also studied, cf. \cite{DL-VPB,Duan-Ukai-Yang-Zhao,Huang-Wang-Yang-2010-1,Huang-Wang-Yang-2010-2,Huang-Wang-Wang-Yang-2013,Huang-Wang-Wang-Yang-2016,
Liu-Yang-Yu-Zhao-2006,Liu-Yu-2004,Xin-Zeng-2010,Yang-Zhao-2006,Yang-Zhao-2005,Yu-2005} and the references cited therein.

Recently, a new model called the INSF system with viscous heating was derived by Bardos-Levermore-Ukai-Yang \cite{Bardos-Levermore-Ukai-Yang-2008}.
The aim of the present paper is to justify such an incompressible hydrodynamic approximation to the Boltzmann equation in a periodic box via an $L^2-L^\infty$ method developed in \cite{EGKM-13,EGKM-15,Guo-2010,Guo-Jang-2010,Guo-Jang-Jiang-2009,Guo-Jang-Jiang-2010}.
We now outline a few key points of the paper which are distinct to some extent with the previous work by Bardos-Levermore-Ukai-Yang \cite{Bardos-Levermore-Ukai-Yang-2008}:
\begin{itemize}
\item The odd-even decomposition of the rescaled Boltzmann equation is more complicate and accurate, namely, we expand the solution of the Boltzmann equation up to sixth order with a remainder.

\item To determine the diffusive coefficients, we introduce the super Burnett functions which play a vital role in defining the macroscopic parts of the diffusive coefficients.

\item A good structure of the INSF system with viscous heating is observed so that the smooth solutions of the macroscopic equations
are obtained via an elementary energy method.

\item We design an elaborate space $\FX_\de$ to capture the properties of the solution of the remainder equation in $L^2\cap L^\infty$ setting.
\end{itemize}

The organization of the paper is as follows. Section \ref{pre} contains some elementary identities and estimates regarding the Boltzmann collision operators. We provide a direct approach to derive the INSF equations with viscous heating and present the construction of the diffusive coefficients in Section \ref{INSFC}. Sections \ref{l2leq} and \ref{lifleq} are devoted to the $L^2$ and $L^\infty$ estimates of the linear equation of the remainder, respectively. The proof of our main result Theorem \ref{mre} is concluded in Section \ref{proof}.

\subsection{Notations and Norms}
Throughout this paper,  $C$ denotes some generic positive (generally large) constant and $\la,\la_1,\la_2$ as well as $\la_0$ denote some generic positive (generally small) constants, where $C$ may take different values in different places. $D\lesssim E$ means that  there is a generic constant $C>0$
such that $D\leq CE$. $D\sim E$
means $D\lesssim E$ and $E\lesssim D$.
 Let $1\leq p\leq \infty$, we denote $\Vert \,\cdot \,\Vert _{p }$ either the $L^{p }(\T^3
\times \R^{3})-$norm or the $L^{p }(\T^3 )-$norm, and denote $\|\cdot \|_{\nu}\equiv \|\nu^{1/2}\cdot
\|_2$. Moreover,
$(\cdot,\cdot)$ denotes the $L^{2}$ inner product in
$\T^3\times {\R}^{3}$ or $\T^3$  with
the $L^{2}$ norm $\|\cdot\|_2$, and $\langle\cdot,\cdot\rangle$ stands for the $L^{2}$ inner product in $\R^3_v$.

\section{Preliminary}\label{pre}
 In this section, we give some basic identities and significant estimates which will be used in the later proofs. The first one is concerned with the relations between the nonlinear operator $\Ga$ and linear operator $L.$
\begin{lemma}\label{NL}
It holds that
\begin{equation}\label{ga2-3}
\Ga(\FP g,\FP g)=\frac{1}{2}L\left\{\frac{(\FP g)^2}{\sqrt{\mu}}\right\}, \ \
\Ga(\FP g,(\FP g)^2\mu^{-1/2})+\Ga((\FP g)^2\mu^{-1/2},\FP g)=\frac{1}{3}L\left\{\frac{(\FP g)^3}{\mu}\right\}.
\end{equation}
\end{lemma}
\begin{proof}
The first identity in \eqref{ga2-3} has been proved in \cite[pp.648-649]{Guo-2006}, the second one can be verified similarly, we omit the details for simplicity. This completes the proof of Lemma \ref{NL}.
\end{proof}
The following significant relations are quoted form \cite[Lemma 4.4, pp.711]{Bardos-Golse-Levermore-1993} and \cite[Propostion 2.5, pp.17]{Bardos-Levermore-Ukai-Yang-2008} as well as \cite[(2.36)-(2.36), pp.16-17]{Bardos-Levermore-Ukai-Yang-2008}.
\begin{lemma}\label{Burt}
It holds that
\begin{equation}\label{AA}
\begin{split}
\left\langle A_{ij}(v), L^{-1}A_{kl}(v)\right\rangle=&\frac{1}{10}\left\langle A(v): L^{-1}A(v)\right\rangle\left(\de_{ik}\de_{jl}+\de_{il}\de_{jk}-\frac{2}{3}\de_{ij}\de_{kl}\right)
\\=&\mu_\ast\left(\de_{ik}\de_{jl}+\de_{il}\de_{jk}-\frac{2}{3}\de_{ij}\de_{kl}\right),
\end{split}
\end{equation}
\begin{equation}\label{BB}
\begin{split}
\left\langle B_i(v), L^{-1}B_j(v)\right\rangle=\frac{1}{3}\left\langle B(v)\cdot L^{-1}B(v)\right\rangle\de_{ij}
=\ka_\ast\de_{ij},
\end{split}
\end{equation}
\begin{equation}\label{AtB}
\begin{split}
\left\langle A_{ij}(v), v_kL^{-1}B_l(v)\right\rangle-\left\langle A_{ik}(v), v_jL^{-1}B_l(v)\right\rangle
=\frac{2}{3}\ka_\ast(\de_{ik}\de_{jl}-\de_{ij}\de_{kl}),
\end{split}
\end{equation}
and
\begin{equation}\label{AB}
\begin{split}
\left\langle L^{-1}A_{ij}(v), v_kL^{-1}B_{l}(v)\right\rangle=&\frac{1}{10}\left\langle L^{-1}A(v): v\otimes L^{-1}B(v)\right\rangle\left(\de_{ik}\de_{jl}+\de_{il}\de_{jk}-\frac{2}{3}\de_{ij}\de_{kl}\right).
\end{split}
\end{equation}
In addition,
it follows
\begin{equation}\label{ABv}
\begin{split}
&\left\langle\Ga\left(v_i\sqrt{\mu},L^{-1}A_{kl}(v)\right)+\Ga\left(L^{-1}A_{kl}(v),v_i\sqrt{\mu}\right), L^{-1}B_j(v)\right\rangle
+\left\langle  v_iA_{kl}(v), L^{-1}B_j(v)\right\rangle\\
&\qquad=\left\langle A_{ij}(v): L^{-1}A_{kl}(v)\right\rangle.
\end{split}
\end{equation}
\end{lemma}
Let us now report the following result which can be directly proved by the definition of $\Ga$.
\begin{lemma}\label{gadec}
Let $p_1(v)$ and $p_2(v)$ be any polynomials in $v$, then for any functions $a(t,x)$ and $b(t,x)$, there exist constants  $c_1,c_2\in(0,1/4)$  such that
$$
|ab|\mu^{c_2}\lesssim |\Ga(a p_1(v)\sqrt{\mu},bp_2(v)\sqrt{\mu})|\lesssim|ab|\mu^{c_1}.
$$
\end{lemma}
The following lemma is devoted to the $L^p$ estimates of the nonlinear operator $\Ga$.
\begin{lemma}\label{es.nop}
It holds that for $l\geq0$,
\begin{equation}\label{Ga.lif}
\|\nu^{-1}w_{l}\Ga(f_1,f_2)\|_\infty\leq C\|w_{l}f_1\|_{\infty}\|w_{l}f_2\|_{\infty},
\end{equation}
and for $l>\frac{3}{2}$,
\begin{equation}\label{Ga.l2}
\|\nu^{-1/2}(\Ga(f_1,f_2)+\Ga(f_2,f_1))\|^2_{2}\leq C\|w_l\nu f_1\|^2_{\infty}\|f_2\|^2_{\nu}.
\end{equation}
In particular, it holds that
\begin{equation}\label{Ga.l22}
\|\nu^{-1/2}\Ga(f,f)\|^2_{2}\leq C\|w_lf\|^2_{\infty}\|f\|^2_{\nu},\ \text{for}\ l>3/2.
\end{equation}
\end{lemma}
\begin{proof}
The proof of \eqref{Ga.lif} has been given in \cite[Lemma 5, pp.730]{Guo-2010}, and the proofs for \eqref{Ga.l2} and \eqref{Ga.l22} are similar as that of \cite[Lemma 2.3, pp.12]{LY-2016}.
\end{proof}

Recall the definition for $K$ in \eqref{defK}, one can rewrite
\begin{equation}\label{Kk}
[Kf](v)=\int_{\R^3}k(v,v')f(v')dv'=\int_{\R^3}[k_2(v,v')-k_1(v,v')]f(v')dv',
\end{equation}
with
$$
k_1(v,v')=\int_{\S^2}|(v-v')\cdot \omega|\sqrt{\mu(v)}\sqrt{\mu(v')}d\omega,
$$
and
$$
|k_2(v,v')|=C|v-v'|^{-1}\exp\left(-\frac{1}{8}|v-v'|^2-\frac{1}{8}\frac{(|v|^2-|v'|^2)^2}{|v-v'|^2}\right).
$$
The following lemma which states the estimates of $k(v,v')$ is borrowed from Lemma 3 of \cite[pp.727]{Guo-2010} and \cite[Lemma 3.3.1, pp.49]{Glassey-1996}.
\begin{lemma}\label{es.k}
It holds that
 \begin{equation*}\label{K.ip1}
\int_{\R^3}k(v,v')dv'\leq \frac{C}{1+|v|},
\end{equation*}
and moreover, for any $l\geq0$,
\begin{equation*}\label{K.ip4}
w_l(v)\int_{\R^3}k(v,v')\frac{e^{\varepsilon|v-v'|^2}}{w_l(v')}dv'\leq \frac{C}{1+|v|},
\end{equation*}
where $\vps>0$ and sufficiently small. 
\end{lemma}
Finally, we cite the $L^p-L^{q}-$estimate on the Riesz potential
\cite{BO-2013} on torus.
\begin{lemma}\label{risez}
Assume $f\in H^{s}({\T}^{3})$ with $\int_{\T^3}f dx=0$, define
\begin{equation*}
N_{\alpha}(f)=\Delta^{-\frac{\alpha}{2}}\partial^{\ga}_{x}f,
\end{equation*}
if $|\ga|\leq\alpha<3$,\ $p>1$, then we have
\begin{equation*}
\|N_{\alpha}(f)\|_{L^{q}}\leq C\|f\|_{L^{p}}.
\end{equation*}
Here $\frac{1}{q}=\frac{1}{p}-\frac{\alpha-|\ga|}{3}$.
\end{lemma}

\section{INSF equations with viscous heating and diffusive coefficients}\label{INSFC}
One purpose of this section is to show the derivation of the INSF equations \eqref{INSF} and \eqref{INSF2}, although a formal
one has been given in \cite[Section 3, pp.19]{Bardos-Levermore-Ukai-Yang-2008}, here we will propose a more direct approach to derive the INSF equations \eqref{INSF}.
In addition, the $H^s(\T^3)$ estimates of the solutions of the system \eqref{INSF} and \eqref{INSF2} as well as the estimates for the coefficients $f_1, f_2,\cdots,f_6$ will also be deduced.
\subsection{Derivation of INSF with viscous heating}
Let us now derive the equations $\eqref{INSF}$. Taking the inner product of \eqref{f2e}, \eqref{f3e} and \eqref{f4e} with $[v\sqrt{\mu},v\sqrt{\mu},\frac{|v|^2-5}{2}\sqrt{\mu}]$ with respect to $v$ over $\R^3$, respectively, one has
\begin{equation*}\label{f0ip}
\langle v\cdot \na_xf_1,v\sqrt{\mu}\rangle=0,
\end{equation*}
\begin{equation*}\label{f1ip}
\langle \pa_tf_1+v\cdot \na_xf_2,v\sqrt{\mu}\rangle=0,
\end{equation*}
\begin{equation*}\label{f2ip}
\left\langle \pa_tf_2+v\cdot \na_xf_3,\frac{|v|^2-5}{2}\sqrt{\mu}\right\rangle=0.
\end{equation*}
Substituting $\eqref{f1-5}_1$, $\eqref{f1-5}_2$  and $\eqref{f1-5}_3$ into the above equations, we further obtain
$$
\na_x\cdot u_1=0,
$$
\begin{equation}\label{ue}
\begin{split}
\pa_t u_1+&\left \langle v\cdot \na_x\left\{\rho_1+\frac{1}{6}(|u_1|^2+3\ta_1)(|v|^2-3)\right\}\sqrt{\mu},v\sqrt{\mu}\right\rangle\\
&-\na_x\cdot\left\langle A(v):\na_xu_1,L^{-1}A(v)\right\rangle+\frac{1}{2}\na_x\cdot\left\langle A(v):u_1\otimes u_1,A(v)\right\rangle=0,
\end{split}
\end{equation}
and
\begin{equation}\label{tae}
\begin{split}
\pa_t &\left \langle\left\{\rho_1+\frac{1}{6}(|u_1|^2+3\ta_1)(|v|^2-3)\right\}\sqrt{\mu},\frac{|v|^2-5}{2}\sqrt{\mu}\right\rangle
\\&+\na_x\cdot\left \langle -v\cdot\na_xf_2+\Gamma(f_1,f_2)+\Gamma(f_2,f_1),L^{-1}B(v)\right\rangle=0,
\end{split}
\end{equation}
respectively.

Next, by applying \eqref{AA} in Lemma \ref{Burt}, we see that \eqref{ue} is equivalent to
\begin{equation}\label{ueq}
\pa_tu_1+u_1\cdot\na_xu_1+\na_xp_1=\mu_\ast\Delta u_1,
\end{equation}
with $p_1=\rho_1+\ta_1.$

As to \eqref{tae}, the first term on the left hand side gives rise to
\begin{equation}\label{tta}
\pa_t\left(\frac{1}{2}|u_1|^2+\frac{3}{2}\ta_1-\rho_1\right).
\end{equation}
We now calculate $\na_x\cdot\left \langle -v\cdot\na_xf_2,L^{-1}B(v)\right\rangle$ and $\na_x\cdot\left \langle \Gamma(f_1,f_2)+\Gamma(f_2,f_1),L^{-1}B(v)\right\rangle$ as follows.
One can see that
$$\na_x\cdot\left \langle v\cdot \na_x\left\{L^{-1}A(v):\na_xu_1\right\},L^{-1}B(v)\right\rangle=0.$$
Indeed, from \eqref{AB}, it follows
\begin{equation*}
\begin{split}
\sum\limits_{i,j,k,l=1}^3\pa_{kl}&\left\langle v_k\left\{L^{-1}A_{ij}(v)\pa_iu_1^j\right\},L^{-1}B_l(v)\right\rangle\\
=&
\sum\limits_{i,j,k,l=1}^3\frac{1}{10}\left\langle L^{-1}A(v): v\otimes L^{-1}B(v)\right\rangle\left(\de_{ik}\de_{jl}+\de_{il}\de_{jk}-\frac{2}{3}\de_{ij}\de_{kl}\right)\pa_{kli}u_1^j=0.
\end{split}
\end{equation*}
Consequently, we have by applying
\eqref{BB} that
\begin{equation}\label{ta1}
\begin{split}
\na_x\cdot&\left \langle -v\cdot\na_xf_2,L^{-1}B(v)\right\rangle\\
=&-\na_x\cdot\left \langle v\cdot \na_x\left\{\rho_1+\frac{1}{6}(|u_1|^2+3\ta_1)(|v|^2-3)\right\}\sqrt{\mu},L^{-1}B(v)\right\rangle\\
&-\na_x\cdot\left \langle v\cdot \na_x\left\{ -L^{-1}A(v):\na_xu_1+\frac{1}{2} A(v):u_1\otimes u_1\right\},L^{-1}B(v)\right\rangle\\
=&-\ka_\ast\Delta\ta_1-\frac{\ka_\ast}{3}\Delta (|u_1|^2)-\frac{1}{2}\na_x\cdot\left \langle v\cdot \na_x\left\{ A(v):u_1\otimes u_1\right\},L^{-1}B(v)\right\rangle.
\end{split}
\end{equation}
By virtue of \eqref{ga2-3}, we now calculate
\begin{equation}\label{ta2}
\begin{split}
\na_x\cdot&\left \langle \Gamma(f_1,f_2)+\Gamma(f_2,f_1),L^{-1}B(v)\right\rangle\\
=&\na_x\cdot\left \langle \Gamma(f_1,\FP f_2)+\Gamma(\FP f_2,f_1),L^{-1}B(v)\right\rangle\\
&+\na_x\cdot\left \langle \Gamma(f_1,\{\FI-\FP\}f_2)+\Gamma(\{\FI-\FP\}f_2,f_1),L^{-1}B(v)\right\rangle\\
=&\na_x\cdot\left \langle u_1\cdot v\left\{\rho_1+\frac{1}{6}(|u_1|^2+3\ta_1)(|v|^2-3)\right\}\sqrt{\mu},B(v)\right\rangle\\
&+\na_x\cdot\left \langle\Gamma\left( u_1\cdot v\sqrt{\mu},\left\{ -L^{-1}A(v):\na_xu_1+\frac{1}{2} A(v):u_1\otimes u_1\right\}\sqrt{\mu}\right),L^{-1}B(v)\right\rangle\\
&+\na_x\cdot\left \langle\Gamma\left(\left\{ -L^{-1}A(v):\na_xu_1+\frac{1}{2} A(v):u_1\otimes u_1\right\},u_1\cdot v\sqrt{\mu}\right),L^{-1}B(v)\right\rangle\\
=&\na_x\cdot\left [ u_1\left(\frac{5}{2}\ta_1+\frac{5}{6}|u_1|^2\right)\right]
+\frac{1}{2} \na_x\cdot\left \langle\Gamma\left( u_1\cdot v\sqrt{\mu},A(v):u_1\otimes u_1\right),L^{-1}B(v)\right\rangle\\
&+\frac{1}{2}\na_x\cdot\left \langle\Gamma\left( A(v):u_1\otimes u_1,u_1\cdot v\sqrt{\mu}\right),L^{-1}B(v)\right\rangle\\
&-\na_x\cdot\left \langle\Gamma\left( u_1\cdot v\sqrt{\mu}, L^{-1}A(v):\na_xu_1\right),L^{-1}B(v)\right\rangle\\
&-\na_x\cdot\left \langle\Gamma\left(L^{-1}A(v):\na_xu_1,u_1\cdot v\sqrt{\mu}\right),L^{-1}B(v)\right\rangle.
\end{split}
\end{equation}
Using the second identity in \eqref{ga2-3}, one sees that
\begin{equation}\label{ta3}
\begin{split}
\frac{1}{2}& \na_x\cdot\left \langle\Gamma\left( u_1\cdot v\sqrt{\mu},A(v):u_1\otimes u_1\right),L^{-1}B(v)\right\rangle
+\frac{1}{2} \na_x\cdot\left \langle\Gamma\left(A(v):u_1\otimes u_1,u_1\cdot v\sqrt{\mu}\right),L^{-1}B(v)\right\rangle\\
&=\frac{1}{6}\na_x\cdot\left \langle L\left(v\otimes v\otimes v:u_1\otimes u_1\otimes u_1\sqrt{\mu}\right),L^{-1}B(v)\right\rangle
-\frac{1}{6}\na_x\cdot\left \langle L\left(u_1\cdot v|v|^2|u_1|^2\sqrt{\mu}\right),L^{-1}B(v)\right\rangle
\\&=-\frac{1}{3}\na_x\cdot(u_1|u_1|^2).
\end{split}
\end{equation}
For the remaining terms in \eqref{ta1} and \eqref{ta2}, we will show that
\begin{equation}\label{ta4}
\begin{split}
\frac{\ka_\ast}{3}&\Delta (|u_1|^2)+\frac{1}{2}\na_x\cdot\left \langle v\cdot \na_x\left\{ A(v):u_1\otimes u_1\right\},L^{-1}B(v)\right\rangle
\\&+\na_x\cdot\left \langle\Gamma\left( u_1\cdot v\sqrt{\mu}, L^{-1}A(v):\na_xu_1\right),L^{-1}B(v)\right\rangle\\
&+\na_x\cdot\left \langle\Gamma\left(L^{-1}A(v):\na_xu_1,u_1\cdot v\sqrt{\mu}\right),L^{-1}B(v)\right\rangle\\
=&\mu_\ast\na_x\cdot\left(u_1(\na_xu_1+(\na_x u_1)^T)\right).
\end{split}
\end{equation}
To confirm this, we first get from \eqref{AtB} that
\begin{equation*}
\begin{split}
\sum\limits_{i,k,l}&\left\{\left\langle A_{ki}(v), v_lL^{-1}B_j(v)\right\rangle-\left\langle A_{kl}(v), v_iL^{-1}B_j(v)\right\rangle \right\}u_i\pa_lu_1^k\\
=&\frac{2\ka_\ast}{3}\sum\limits_{i,k,l}(\de_{ij}\de_{kl}-\de_{ki}\de_{lj})u_1^i\pa_lu_1^k
=\frac{2\ka_\ast}{3}\sum\limits_{k}u_j\pa_ku_1^k-\frac{2\ka_\ast}{3}\sum\limits_{i}u_1^j\pa_ju_1^i=-\frac{2\ka_\ast}{3}\sum\limits_{i}u_1^j
\pa_ju_1^i,
\end{split}
\end{equation*}
Then the left hand side of \eqref{ta4} can be further rewritten as
\begin{equation*}
\begin{split}
-\na_x\cdot&\left\{\sum\limits_{i,k,l}\left\{\left\langle A_{ki}(v), v_lL^{-1}B(v)\right\rangle-\left\langle A_{kl}(v), v_iL^{-1}B(v)\right\rangle \right\}u_1^i\pa_lu_1^k\right\}
\\&+\na_x\cdot\left\{\sum\limits_{i,k,l}\left\langle A_{ki}(v), v_lL^{-1}B(v)\right\rangle u_1^i\pa_lu_1^k\right\}
\\&+\na_x\cdot\left \langle\Gamma\left( u_1\cdot v\sqrt{\mu}, L^{-1}A(v):\na_xu_1\right),L^{-1}B(v)\right\rangle\\
&+\na_x\cdot\left \langle\Gamma\left(L^{-1}A(v):\na_xu_1,u_1\cdot v\sqrt{\mu}\right),L^{-1}B(v)\right\rangle\\
=&\na_x\cdot\left\{\sum\limits_{i,k,l}\left\langle A_{kl}(v), v_iL^{-1}B(v)\right\rangle u_1^i\pa_lu_1^k\right\}\\&+\na_x\cdot\left \langle\Gamma\left( u_1\cdot v\sqrt{\mu}, L^{-1}A(v):\na_xu_1\right),L^{-1}B(v)\right\rangle
\\&+\na_x\cdot\left \langle\Gamma\left( u_1\cdot v\sqrt{\mu}, L^{-1}A(v):\na_xu_1\right),L^{-1}B(v)\right\rangle\\
=&\sum\limits_{j} \pa_j\left\{u_1^i\langle A_{ij}(v), L^{-1}A_{kl}(v)\rangle\pa_lu_1^k\right\},
\end{split}
\end{equation*}
according to \eqref{ABv}.
Hence \eqref{ta4} is valid. We now conclude from \eqref{tta}, \eqref{ta1}, \eqref{ta2}, \eqref{ta3} and \eqref{ta4} that
\begin{equation}\label{taeq}
\pa_t\left(\frac{1}{2}|u_1|^2+\frac{3}{2}\ta_1-\rho_1\right)+\na_x\cdot\left [ u_1\left(\frac{5}{2}\ta_1+\frac{1}{2}|u_1|^2\right)\right]
=\ka_\ast\Delta\ta_1+\mu_\ast\na_x\cdot\left(u_1(\na_xu_1+(\na_x u_1)^T)\right).
\end{equation}
Lastly, the subtraction of \eqref{taeq} and $u_1\cdot\eqref{ueq}$ yields the energy equation $\eqref{INSF}_3$.

Likewise, one can see that \eqref{INSF2} follows from the inner products $\langle\eqref{f4e},\sqrt{\mu}\rangle$, $\langle\eqref{f5e},v\sqrt{\mu}\rangle$, $\langle\eqref{f6e},\frac{v^2-5}{2}\sqrt{\mu}\rangle$ and $\langle\eqref{f6e},\sqrt{\mu}\rangle$, we refer to \cite[Section 4, pp.643]{Guo-2006} for more details.

\subsection{Diffusive coefficients}
The diffusive coefficients $f_1,f_2,\cdots,f_6$ will be determined by solving the Cauchy problem \eqref{INSF} and \eqref{INSF2} on torus $\T^3.$

\begin{proposition}\label{ss}
Assume the condition $(\CA_1)$ in Theorem \ref{mre} is valid,
then there exist unique functions $f_1,f_2,\cdots,f_6$ with zero mean hydrodynamic fields, which read
$$
(f_{r}(t,x,v),[1,v,(v^2-3)]\sqrt{\mu})=0,\ r=1,2,\cdots,6,
$$
such that $f_1,f_2,\cdots,f_6$ satisfy \eqref{INSF}, \eqref{INSF2} and $\eqref{f1-5}$.
Moreover, for $1\leq r \leq6$, and for any $s\geq 2$ and $l\geq0$, there exists $\la_0>0$ and a polynomial $P$ with $P(0)=0$ such that
\begin{equation*}
\begin{split}
&\sum\limits_{\al_0+\al\leq s\atop{\al_0\leq s-1}}\|\pa_t^{\al_0}\na_x^{\al}f_r(t)\|_2\leq e^{-\la_0t}
P(
\|[u_{1,0},\ta_{1,0}]\|_{H^{2s+2(r-1)}}+\|[u_{2,0},\ta_{2,0}]\|_{H^{2s}}),\ 1\leq r\leq 3,\\
&\sum\limits_{\al_0+\al\leq s\atop{\al_0\leq s-1}}\|\pa_t^{\al_0}\na_x^{\al}f_r(t)\|_2\leq e^{-\la_0t}
P(
\|[u_{1,0},\ta_{1,0}]\|_{H^{2s+2r}}+\|[u_{2,0},\ta_{2,0}]\|_{H^{2s+2(r-1)}}),\ 4\leq r\leq 6,
\end{split}
\end{equation*}
\begin{equation*}
\begin{split}
\sum\limits_{\al_0+\al\leq s\atop{\al_0\leq s-1}}\|w_l\pa_t^{\al_0}\na_x^{\al}f_r(t)\|_\infty\leq e^{-\la_0t} P(\|[u_{1,0},\ta_{1,0}]\|_{H^{2s+2r}}+\|[u_{2,0},\ta_{2,0}]\|_{H^{2s+2}}),\ 1\leq r\leq 3,
\end{split}
\end{equation*}
and
\begin{equation}\label{ssol}
\begin{split}
&\sum\limits_{\al_0+\al\leq s\atop{\al_0\leq s-1}}\|w_l\pa_t^{\al_0}\na_x^{\al}f_r(t)\|_\infty\leq e^{-\la_0t} P(\|[u_{1,0},\ta_{1,0}]\|_{H^{2s+2r+2}}+\|[u_{2,0},\ta_{2,0}]\|_{H^{2s+2r}}),\ 4\leq r\leq 6.
\end{split}
\end{equation}
\end{proposition}
\begin{proof}
To determine $f_1$ and $f_2$, we start by solving the system \eqref{INSF}.
The basic tool in the proof of the existence of \eqref{INSF} is the Galerkin approximation and the classical energy method \cite{Temam-1979,Matsumura-Nishida-1980}, we skip the details for brevity.
In what follows,
we first show that
\begin{equation}\label{1de}
\begin{split}
\sum\limits_{\al_0+\al\leq2}\left\|\pa_t^{\al_0}\na_x^\al u_1\right\|_2
+\sum\limits_{\al_0+\al\leq2,\al_0\leq1}\left\|\pa_t^{\al_0}\na_x^\al[\rho_1,\ta_1]\right\|_2\leq Ce^{-\la_0t}P\left(\|[u_{1,0},\ta_{1,0}]\|_{H^{4}}\right)
\end{split}
\end{equation}
under the {\it a priori} assumption
\begin{equation}\label{apss}
\sum\limits_{\al_0+\al\leq2}\left\|\pa_t^{\al_0}\na_x^\al u_1\right\|_2
+\sum\limits_{\al_0+\al\leq2,\al_0\leq1}\left\|\pa_t^{\al_0}\na_x^\al[\rho_1,\ta_1]\right\|_2\leq \vps_0.
\end{equation}
Taking the inner product of $\pa_t^{\al_0}\na_x^\al\eqref{INSF}_2$ with $\pa_t^{\al_0}\na_x^\al u_1$ $(\al_0+\al\leq2)$ over $\T^3$ and employing $\na_x\cdot u_1=0$ and Sobolev's inequality, one has
\begin{equation}\label{su}
\begin{split}
\frac{1}{2}\frac{d}{dt}\|\pa_t^{\al_0}\na_x^\al u_1\|^2_2+\mu_\ast\|\na_x \pa_t^{\al_0}\na_x^\al u_1\|^2_2\leq&
\sum\limits_{\al_0'\leq \al_0,\al'\leq \al}|(\pa_t^{\al'_0}\na_x^{\al'}u_1\cdot\na_x\pa_t^{\al_0-\al_0'}\na_x^{\al-\al'} u_1,\pa_t^{\al_0}\na_x^\al u_1)|\\
 \leq& C\sum\limits_{\al'_0+\al'\leq2}\left\{\|\pa_t^{\al'_0}\na_x^{\al'} u_1\|_{2}\|\na_x \pa_t^{\al'_0}\na^{\al'} u_1\|_2\right\}\| \pa_t^{\al_0}\na_x^\al u_1\|_2.
\end{split}
\end{equation}
Taking the summation of \eqref{su} over $\al_0+\al\leq2$ and using \eqref{apss}, we obtain for some $\la_1>0$
\begin{equation}\label{sud0}
\begin{split}
\sum\limits_{\al_0+\al\leq2}\frac{d}{dt}\|\pa_t^{\al_0}\na_x^\al u_1\|^2_2+2\la_1\sum\limits_{\al_0+\al\leq2}\| \pa_t^{\al_0}\na_x^\al u_1\|^2_2\leq 0,
\end{split}
\end{equation}
where we have also used the following Sobolev inequality on torus
\begin{equation}\label{sob}
\|u\|_p\leq C\|\na_x u\|_2\ \  \text{with}\ \ p\in[2,6]\ \ \text{for} \ \ u\in H^1(\T^3)\ \text{and}\ \int_{\T^3}udx=0.
\end{equation}
\eqref{sud0} further implies
\begin{equation}\label{sud}
\begin{split}
\sum\limits_{\al_0+\al\leq2}\|\pa_t^{\al_0}\na_x^\al u_1\|_2
 \leq& Ce^{-\la_1t}\sum\limits_{\al_0+\al\leq2}\|\pa_t^{\al_0}\na_x^\al u_{1,0}\|_2.
\end{split}
\end{equation}
We
note immediately that the the temporal derivatives of the above initial data
are understood by the equations $\eqref{INSF}_2$, for instance,
$\pa_t\pa_x^\al u_{1,0}$
is defined as
\begin{equation}\label{tu1dec}
\lim\limits_{t\to 0_+}\pa_t\na_x^\al u_1(t,x)=\lim\limits_{t\to 0_+}\na_x^\al\left\{-\FP_0(u_1\cdot\na_xu_1)+\mu_\ast\Delta u_1\right\}
=\na_x^\al\left\{-\FP_0(u_{1,0}\cdot\na_xu_{1,0})+\mu_\ast\Delta u_{1,0}\right\}.
\end{equation}
With this, one
can formally view the two spatial derivatives as ``equivalent" to one temporal derivative
and therefore there exists a polynomial $P$ with $P(0)=0$ such that
\begin{equation*}
\sum\limits_{\al_0+\al\leq2}\|\pa_t^{\al_0}\na_x^\al u_{1,0}\|_2\leq P(\|u_{1,0}\|_{H^4}).
\end{equation*}
As to the estimates for $\ta_1$ and $\rho_1$,
we first get from equation $\eqref{INSF}_2$ that
\begin{equation}\label{rhoex}
\begin{split}
\rho_1(t,x)=-\Delta^{-1}\na_x\cdot(u_1\cdot \na_xu_1)-\ta_1.
\end{split}
\end{equation}
Inserting \eqref{rhoex} into $\eqref{INSF}_3$, one has
\begin{equation}\label{ta12}
\begin{split}
\frac{5}{2}\pa_t\ta_1&+\pa_t\Delta^{-1}\na_x\cdot(u_1\cdot \na_xu_1)+u_1\cdot\na_x\left(\frac{5}{2}\ta_1+\Delta^{-1}\na_x\cdot(u_1\cdot \na_xu_1)\right)\\
=&\ka_\ast\Delta\theta_1+\frac{1}{2}\mu_\ast
\left|\na_xu_1+(\na_xu_1)^T\right|^2.
\end{split}
\end{equation}
We now gets from the inner product of $\pa_t^{\al_0}\na_x^\al\eqref{ta12}$ $(\al_0+\al\leq 2, \al_0\leq 1)$ with $\pa_t^{\al_0}\na_x^\al\ta_1$ over $\T^3$ that
\begin{equation*}\label{ta1p3}
\begin{split}
\frac{5}{2}(\pa_t&\pa_t^{\al_0}\na_x^\al\ta_1,\pa_t^{\al_0}\na_x^\al\ta_1)+(\pa_t\pa_t^{\al_0}\na_x^\al\Delta^{-1}\na_x\cdot(u_1\cdot \na_xu_1),\pa_t^{\al_0}\na_x^\al\ta_1)\\&+(\pa_t^{\al_0}\na_x^\al[u_1\cdot\na_x(5/2\ta_1+\Delta^{-1}\na_x\cdot(u_1\cdot \na_xu_1))],\pa_t^{\al_0}\na_x^\al\ta_1)\\
=&\ka_\ast(\pa_t^{\al_0}\na_x^\al\Delta_x \ta_1,\pa_t^{\al_0}\na_x^\al\ta_1)
+\frac{\mu_\ast}{2}\left(\pa_t^{\al_0}\na_x^\al
\left|\na_xu_1+(\na_xu_1)^T\right|^2,\pa_t^{\al_0}\na_x^\al\ta_1\right).
\end{split}
\end{equation*}
By integration by parts and using \eqref{sob} and Lemma \eqref{risez}, one has
\begin{equation}\label{ta1p4}
\begin{split}
\frac{d}{dt}&\sum\limits_{\al_0+\al\leq 2,\al_0\leq 1}\|\pa_t^{\al_0}\na_x^\al\ta_1\|_2^2
+\la_2\sum\limits_{\al_0+\al\leq 2,\al_0\leq 1}\|\na_x\pa_t^{\al_0}\na_x^\al\ta_1\|_2^2\\
\leq& C\sum\limits_{\al_0+\al\leq2}\|\pa_t^{\al_0}\na_x^\al u_1\|_{2}^2\sum\limits_{\al_0+\al\leq 2}\|\pa_t^{\al_0}\na_x^\al u_1\|_{H^1}^2,
\end{split}
\end{equation}
for some $\la_2>0.$
On the other hand, \eqref{taeq} and \eqref{meanm} imply
$$
\int_{\T^3}(3\ta_1+|u_1|^2)dx=0,
$$
using this and Poincar$\acute{e}$'s inequality, one further gets from \eqref{ta1p4} that
\begin{equation*}\label{ta1p41}
\begin{split}
\frac{d}{dt}&\sum\limits_{\al_0+\al\leq2,\al_0\leq1}\|\pa_t^{\al_0}\na_x^\al\ta_1\|_2^2
+2\la_2\sum\limits_{\al_0+\al\leq2,\al_0\leq1}\|\pa_t^{\al_0}\na_x^\al\ta_1\|_2^2\\
\leq& C\sum\limits_{\al_0+\al\leq2}\|\pa_t^{\al_0}\na_x^\al u_1\|_{2}^2\sum\limits_{\al_0+\al\leq 2}\|\pa_t^{\al_0}\na_x^\al u_1\|_{H^1}^2.
\end{split}
\end{equation*}
As a sequence, according to \eqref{su}, it follows that
\begin{equation}\label{ta1p6}
\begin{split}
\sum\limits_{\al_0+\al\leq2,\al_0\leq 1}\|\pa_t^{\al_0}\na_x^\al\ta_1\|_2^2\leq& e^{-2\la_2t}\sum\limits_{\al_0+\al\leq2,\al_0\leq 1}\|\pa_t^{\al_0}\na_x^\al\ta_{1,0}\|_2^2
\\&+P^2(\|u_{1,0}\|_{H^4})e^{-2\la_2t}\int_{0}^te^{(2\la_2-2\la_1)s} \|\pa_s^{\al_0}\na_x^\al u_1\|_{H^1}^2ds\\
\leq& e^{-2\la_2t}\sum\limits_{\al_0+\al\leq2,\al_0\leq 1}\|\pa_t^{\al_0}\na_x^\al\ta_{1,0}\|_2^2
+e^{-2\la_2t}P^2(\|u_{1,0}\|_{H^4}),
\end{split}
\end{equation}
provided $0<\la_2< \la_1$.

Moreover, \eqref{rhoex} gives
\begin{equation}\label{ta1p5}
\begin{split}
\sum\limits_{\al_0+\al\leq2,\al_0\leq1}\|\pa_t^{\al_0}\na_x^\al\rho_1\|_2^2\leq \vps_0\sum\limits_{\al_0+\al\leq2}\|\pa_t^{\al_0}\na_x^\al u_1\|^2_{2}
+C\sum\limits_{\al_0+\al\leq2,\al_0\leq1}\|\pa_t^{\al_0}\na_x^\al\ta_1\|_2^2.
\end{split}
\end{equation}
As \eqref{tu1dec}, we define
\begin{equation*}\label{tta101}
\begin{split}
\pa_t\na_x^\al\ta_{1,0}=&-\frac{2}{5}\lim\limits_{t\to 0_+}\pa_t\na_x^\al\Delta^{-1}\na_x\cdot(u_1\cdot \na_xu_1)
+\frac{2}{5}\na_x^{\al}\lim\limits_{t\to 0_+}\left\{\ka_\ast\Delta\theta_1+\frac{1}{2}\mu_\ast
\left|\na_xu_1+(\na_xu_1)^T\right|^2\right\}\\&-
\frac{2}{5}\lim\limits_{t\to 0_+}\na_x^\al\left\{u_1\cdot\na_x\left(\frac{5}{2}\ta_1+\Delta^{-1}\na_x\cdot(u_1\cdot \na_xu_1)\right)\right\}.
\end{split}
\end{equation*}
Consequently,
\begin{equation}\label{tta102}
\begin{split}
\sum\limits_{\al_0+\al\leq2,\al_0\leq1}\|\pa_t^{\al_0}\na_x^\al\ta_{1,0}\|_2\leq P(\|[u_{1,0},\ta_{1,0}]\|_{H^4}).
\end{split}
\end{equation}
Letting $\la_2<\la_1$ and taking $\la_0=\min\{\la_2,\la_1\}$, we thereby obtain from \eqref{ta1p6}, \eqref{ta1p5} and \eqref{tta102}
that
\begin{equation}\label{ta1p7}
\begin{split}
\sum\limits_{\al_0+\al\leq2,\al_0\leq1}\|\pa_t^{\al_0}\na_x^\al[\rho_1,\ta_1]\|_2\leq& e^{-\la_0t}P(\|[u_{1,0},\ta_{1,0}]\|_{H^4}).
\end{split}
\end{equation}
Therefore, \eqref{1de} follows from \eqref{sud} and \eqref{ta1p7}. It should be pointed out that we can obtain the exponential decay of $\|\pa_t^2u_1\|_2$ but $\|\pa_t^2[\rho_1,\ta_1]\|_2$, this phenomenon is essentially determined by the different structure of the equations
$\eqref{INSF}_2$ and \eqref{ta12}. In what follows, we shall not include the highest time-derivatives of $u_1$ or $u_2$ in order to make our presentation more easy to read.

With \eqref{1de} in hand, we now
turn to deduce the following higher order $L^2$ estimates:
\begin{equation}\label{1hde}
\begin{split}
\sum\limits_{\al_0+\al\leq s,\al_0\leq s-1}\left\|\pa_t^{\al_0}\na_x^\al[\rho_1,u_1,\ta_1]\right\|_2\leq Ce^{-\la_0t}P\left(\|[u_{1,0},\ta_{1,0}]\|_{H^{2s}}\right),\ s\geq3.
\end{split}
\end{equation}
Since $\sum\limits_{\al_0+\al= s,\al_0\leq s-1}\left\|\pa_t^{\al_0}\na_x^\al[\rho_1,u_1,\ta_1]\right\|_2$ may not be small for $s\geq3$, we need to proceed differently.
From $\eqref{INSF}_2$, it follows for $\al_0+\al=s\geq3$, $\al_0\leq1$ and $\eta>0$
\begin{equation*}\label{suh}
\begin{split}
\frac{1}{2}\frac{d}{dt}&\sum\limits_{\al_0+\al=s}\|\pa_t^{\al_0}\na_x^\al u_1\|^2_2
+\mu_\ast\sum\limits_{\al_0+\al=s}\|\na_x \pa_t^{\al_0}\na_x^\al u_1\|^2_2\\
\leq&
\sum\limits_{\al_0'\leq \al_0,\al'\leq \al}\sum\limits_{\al_0+\al=s}|(\pa_t^{\al'_0}\na_x^{\al'}u_1\otimes\pa_t^{\al_0-\al_0'}\na_x^{\al-\al'} u_1,\na_x\pa_t^{\al_0}\na_x^\al u_1)|\\
\leq& \sum\limits_{\al'_0+\al'\leq1}\sum\limits_{\al_0+\al=s}\|\pa_t^{\al'_0}\na_x^{\al'}u_1\|_3\|\pa_t^{\al_0-\al_0'}\na_x^{\al-\al'} u_1\|_6\|\na_x\pa_t^{\al_0}\na_x^\al u_1\|_2\\&+\sum\limits_{\al'_0+\al'\geq s-1}\sum\limits_{\al_0+\al=s}\|\pa_t^{\al'_0}\na_x^{\al'}u_1\|_6\|\pa_t^{\al_0-\al_0'}\na_x^{\al-\al'} u_1\|_3\|\na_x\pa_t^{\al_0}\na_x^\al u_1\|_2\\&+\sum\limits_{1<\al'_0+\al'< s-1}\sum\limits_{\al_0+\al=s}\|\pa_t^{\al'_0}\na_x^{\al'}u_1\|_\infty
\|\pa_t^{\al_0-\al_0'}\na_x^{\al-\al'} u_1\|_2\|\na_x\pa_t^{\al_0}\na_x^\al u_1\|_2\\
\leq&\left(\sum\limits_{\al'_0+\al'\leq1 }\|\pa_t^{\al'_0}\na_x^{\al'}u_1\|_{H^1}+\eta\right)\sum\limits_{\al_0+\al=s}\|\na_x \pa_t^{\al_0}\na_x^\al u_1\|^2_2+C_\eta\sum\limits_{\al'_0+\al'\leq s-1}\|\pa_t^{\al'_0}\na_x^{\al'}u_1\|^4_2,
\end{split}
\end{equation*}
where the last inequality holds due to \eqref{sob} and Cauchy-Schwartz's inequality with $\eta>0.$
Then one sees that $u_1$ enjoys the estimates \eqref{1hde},
using a method of induction on $s\geq3$. The corresponding estimates for $\rho_1$ and $\ta_1$ can be obtained in the similar way as deriving \eqref{ta1p7}.

We get at once from \eqref{1hde} and the definitions for $f_1$ and $f_2$ in \eqref{f1-5} that
\begin{eqnarray}\label{f1-f2l2}
\left\{
\begin{array}{rll}
&&\sum\limits_{\al_0+\al\leq s,\al_0\leq s-1}\left\|\pa_t^{\al_0}\na_x^\al f_1\right\|_2
\leq Ce^{-\la_0t}P\left(\|u_{1,0}\|_{H^{2s}}\right),\\[2mm]
&&\sum\limits_{\al_0+\al\leq s,\al_0\leq s-1}\left\|\pa_t^{\al_0}\na_x^\al f_2\right\|_2\leq Ce^{-\la_0t}P\left(\|[u_{1,0},\ta_{1,0}]\|_{H^{2s+2}}\right).
\end{array}\right.
\end{eqnarray}
In addition, from Sobolev's inequality, it follows that
\begin{equation*}\label{f1lif}
\begin{split}
\sum\limits_{\al_0+\al\leq s,\al_0\leq s-1}\left\|w_l\pa_t^{\al_0}\na_x^\al f_1\right\|_\infty\leq Ce^{-\la_0t}P\left(\|u_{1,0}\|_{H^{2s+2}}\right),
\end{split}
\end{equation*}
and
\begin{equation*}\label{f2lif}
\begin{split}
\sum\limits_{\al_0+\al\leq s,\al_0\leq s-1}\left\|w_l\pa_t^{\al_0}\na_x^\al f_2\right\|_\infty\leq Ce^{-\la_0t}P\left(\|[u_{1,0},\ta_{1,0}]\|_{H^{2s+4}}\right),
\end{split}
\end{equation*}
for any $l\geq0.$
Since $L^{-1}$ preserves decay in $v$, c.f. \cite{Caflisch-1980}, one also has by Lemma \ref{gadec} and $\eqref{f1-5}_3$
\begin{equation}\label{f3mi2}
\begin{split}
\sum\limits_{\al_0+\al\leq s,\al_0\leq s-1}\left\|w_l\pa_t^{\al_0}\na_x^\al \{\FI-\FP\}f_3\right\|_2\leq Ce^{-\la_0t}P\left(\|[u_{1,0},\ta_{1,0}]\|_{H^{2s+4}}\right),
\end{split}
\end{equation}
and
\begin{equation*}\label{f3miif}
\begin{split}
\sum\limits_{\al_0+\al\leq s,\al_0\leq s-1}\left\|w_l\pa_t^{\al_0}\na_x^\al\{\FI-\FP\}f_3\right\|_\infty\leq Ce^{-\la_0t}P\left(\|[u_{1,0},\ta_{1,0}]\|_{H^{2s+6}}\right).
\end{split}
\end{equation*}
Let us now turn to estimate $u_2$, $f_4$, $f_5$ and $f_6$. To do this, we first rewrite $\eqref{INSF2}_1$ and $\eqref{INSF2}_2$ as
\begin{eqnarray}\label{INSF3}
\left\{\begin{array}{rlll}
\begin{split}
&\pa_t \rho_1+\na_x\cdot \{\FI-\FP_0\}u_2+\na_x\rho_1\cdot u_1=0,\\
&\pa_t\FP_0u_2+\na_x u_1\cdot \FP_0u_2+u_1\cdot \na_x \FP_0u_2
+\na_x\left(\rho_2+\ta_2-\frac{1}{3}u_1\cdot \FP_0 u_2\right)
\\&\quad=\mu_\ast\Delta \FP_0u_2+\mu_\ast\Delta \{\FI-\FP_0\}u_2-\pa_t\{\FI-\FP_0\}u_2-u_1\na_x\cdot \{\FI-\FP_0\}u_2-\na_x u_1\cdot \{\FI-\FP_0\}u_2
\\&\qquad-u_1\cdot \na_x  \{\FI-\FP_0\}u_2+\frac{1}{3}\na_x\left(u_1\cdot \{\FI-\FP_0\}u_2\right)
\\&\qquad+\frac{\mu_\ast}{3}\na_x\na_x\cdot\{\FI-\FP_0\}u_2
+\na_x\cdot\langle \pa_t f_2, L^{-1}A(v)\rangle
\\&\qquad-\na_x\cdot\langle\Gamma(f_2,f_2),L^{-1}A(v)\rangle
-\na_x\cdot\langle\Gamma(f_1,\{\FI-\FP\}f_3)+\Gamma(\{\FI-\FP\}f_3,f_1),L^{-1}A(v)\rangle
\\&\qquad-\pa_t(\rho_1u_1)-2\na_x\cdot(\rho u_1\otimes u_1)-\na_{x}(\rho_1\ta_1+\frac{5}{2}\ta_1^2-\frac{1}{3}\rho_1|u_1|^2)
\\&\qquad+\mu_\ast\Delta (\rho_1u_1)+\frac{\mu_\ast}{3}\na_x\na_x\cdot(\rho_1u_1),\\
&u_2(0,x)=u_{2,0}(x)=\FP_0u_{2,0}(x).
\end{split}\end{array}\right.
\end{eqnarray}
The proof for the existence of the above {\it linear} system is quite standard.
In what follows we will show that
\begin{equation}\label{2hde}
\begin{split}
\sum\limits_{\al_0+\al\leq s,\al_0\leq s-1}\left\|\pa_t^{\al_0}\na_x^\al u_2\right\|_2\leq Ce^{-\la_0t}P\left(\|[u_{1,0},\ta_{1,0}]\|_{H^{2s+2}}\right),\ s\geq3,
\end{split}
\end{equation}
under the condition \eqref{1hde}.

To begin with, observe that $\eqref{INSF2}_2$ and \eqref{meanm} imply
\begin{equation*}\label{meanu22}
\int_{\T^3}(u_2+\rho_1u_1)dx=\int_{\T^3}(u_{2,0}+\rho_{1,0}u_{1,0})dx=0.
\end{equation*}
Thus we know thanks to \eqref{Divop}
\begin{equation}\label{meanu23}
\int_{\T^3}\FP_0(u_2+\rho_1u_1)dx=\int_{\T^3}\{\FI-\FP_0\}(u_2+\rho_1u_1)dx=0.
\end{equation}
In light of $\eqref{INSF3}_1$ and \eqref{meanu23}, we have by standard elliptic estimates and Poincar$\acute{e}$'s inequality
\begin{equation*}
\begin{split}
\sum\limits_{\al_0+\al\leq s,\al_0\leq s-1}&\left\|\pa_t^{\al_0}\na_x^\al\{\FI-\FP_0\}(u_2+\rho_1u_1)\right\|_2\\
\leq& C\sum\limits_{\al_0+\al\leq s,\al_0\leq s-1}\left\|\pa_t^{\al_0}\na_x^\al\pa_t\rho_1\right\|_2\leq Ce^{-\la_0t}P\left(\|[u_{1,0},\ta_{1,0}]\|_{H^{2s+2}}\right),
\end{split}
\end{equation*}
which further yields
\begin{equation}\label{ndu2}
\begin{split}
\sum\limits_{\al_0+\al\leq s,\al_0\leq s-1}&\left\|\pa_t^{\al_0}\na_x^\al\{\FI-\FP_0\}u_2\right\|_2\\
\leq& C\sum\limits_{\al_0+\al\leq s,\al_0\leq s-1}\left\|\pa_t^{\al_0}\na_x^\al[\pa_t\rho_1,\rho_1,u_1]\right\|_2\leq
C\sum\limits_{\al_0+\al\leq s+1,\al_0\leq s}\left\|\pa_t^{\al_0}\na_x^\al[\rho_1,u_1]\right\|_2\\
\leq &Ce^{-\la_0t}P\left(\|[u_{1,0},\ta_{1,0}]\|_{H^{2s+2}}\right),
\end{split}
\end{equation}
according to \eqref{1hde}.

To prove \eqref{2hde}, it remains now to estimate $\FP_0 u_2$.
Taking the inner product of $\pa_t^{\al_0}\na_x^\al\eqref{INSF3}_2$ with $\na_x \pa_t^{\al_0}\na_x^\al \FP_0 u_2$ and applying Lemma \ref{es.nop} as well as Sobolev's inequality, one has for $\al_0+\al=s$, $\al_0\leq s-1$ and
$s\geq 3$
\begin{equation}\label{su2h}
\begin{split}
\frac{1}{2}\frac{d}{dt}&\sum\limits_{\al_0+\al=s}\|\pa_t^{\al_0}\na_x^\al \FP_0 u_2\|^2_2
+\mu_\ast\sum\limits_{\al_0+\al=s}\|\na_x \pa_t^{\al_0}\na_x^\al \FP_0 u_2\|^2_2\\
\leq&
2\sum\limits_{\al_0'\leq \al_0,\al'\leq \al}\sum\limits_{\al_0+\al=s}|(\pa_t^{\al'_0}\na_x^{\al'}u_1\otimes\pa_t^{\al_0-\al_0'}\na_x^{\al-\al'} \FP_0 u_2,\na_x\pa_t^{\al_0}\na_x^\al \FP_0 u_2)|\\
&+2\sum\limits_{\al_0+\al=s}\left|\left(\pa_t^{\al_0}\na_x^\al\left(u_1\otimes \{\FI-\FP_0\}u_2\right)
,\na_x\pa_t^{\al_0}\na_x^\al \FP_0 u_2\right)\right|\\
&+\sum\limits_{\al_0+\al=s}\left|\left(\pa_t^{\al_0}\na_x^\al\left(\langle \pa_t f_2, L^{-1}A(v)\rangle
-\langle\Gamma(f_2,f_2),L^{-1}A(v)\rangle\right),\na_x\pa_t^{\al_0}\na_x^\al \FP_0 u_2\right)\right|\\
&+\sum\limits_{\al_0+\al=s}\left|\left(\pa_t^{\al_0}\na_x^\al\left(\langle\Gamma(f_1,\{\FI-\FP\}f_3)
+\Gamma(\{\FI-\FP\}f_3,f_1),L^{-1}A(v)\rangle\right),
\na_x\pa_t^{\al_0}\na_x^\al \FP_0 u_2\right)\right|\\
&+\sum\limits_{\al_0+\al=s}\left|\left(\pa_t^{\al_0}\na_x^\al\left[\pa_t(\rho_1u_1)+2\na_x\cdot(\rho u_1\otimes u_1)+\mu_\ast\Delta (\rho_1u_1)\right],\pa_t^{\al_0}\na_x^\al \FP_0 u_2\right)\right|
\\
\leq&\left(\sum\limits_{\al'_0+\al'\leq1 }\|\pa_t^{\al'_0}\na_x^{\al'}u_1\|_{H^1}+\eta\right)\sum\limits_{\al_0+\al=s}\|\na_x \pa_t^{\al_0}\na_x^\al \FP_0 u_2\|^2_2+C_\eta\sum\limits_{\al_0+\al\leq s,\al_0\leq s-1}\|\pa_t^{\al_0}\na_x^{\al}u_1\|^2_2\\
&+C_\eta \sum\limits_{\al_0+\al\leq s,\al_0\leq s-1}\left\|\pa_t^{\al_0}\na_x^\al\{\FI-\FP_0\}u_2\right\|_2^2
+C_\eta\sum\limits_{\al_0+\al\leq s+1,\al_0\leq s}\|\pa_t^{\al_0}\na_x^{\al}u_1\|^2_2\|\pa_t^{\al_0}\na_x^{\al}\rho_1\|^2_2
\\&+C_\eta\sum\limits_{\al_0+\al\leq s+1,\al_0\leq s}\left\|\pa_t^{\al_0}\na_x^\al f_2\right\|^2_2
+C_\eta\sum\limits_{\al_0+\al\leq s}\left\|\pa_t^{\al_0}\na_x^\al f_1\right\|^2_2
\\&+C_\eta\sum\limits_{\al_0+\al\leq s,\al_0\leq s-1}\left\|\pa_t^{\al_0}\na_x^\al \{\FI-\FP\}f_3\right\|^2_2.
\end{split}
\end{equation}
Recalling \eqref{meanu23}, one has
\begin{equation}\label{du2}
\begin{split}
\sum\limits_{\al_0+\al\leq s,\al_0\leq s-1}&\left\|\pa_t^{\al_0}\na_x^\al\FP_0 u_2\right\|_2\\
\leq& \sum\limits_{\al_0+\al\leq s,\al_0\leq s-1}\left\|\pa_t^{\al_0}\na_x^\al\FP_0 (u_2+\rho_1u_1)\right\|_2
+\sum\limits_{\al_0+\al\leq s,\al_0\leq s-1}\left\|\pa_t^{\al_0}\na_x^\al\FP_0 (\rho_1u_1)\right\|_2\\
\leq& C\sum\limits_{\al_0+\al\leq s,\al_0\leq s-1}\left\|\pa_t^{\al_0}\na_x^\al\FP_0 \na_x(u_2+\rho_1u_1)\right\|_2
+\sum\limits_{\al_0+\al\leq s,\al_0\leq s-1}\left\|\pa_t^{\al_0}\na_x^\al\FP_0 (\rho_1u_1)\right\|_2.
\end{split}
\end{equation}
Plugging \eqref{1hde},
\eqref{f1-f2l2}, \eqref{f3mi2}, \eqref{ndu2} and \eqref{du2} into \eqref{su2h} leads us
\begin{equation*}
\begin{split}
\sum\limits_{\al_0+\al\leq s,\al_0\leq s-1}\left\|\pa_t^{\al_0}\na_x^\al\FP_0 u_2\right\|_2
\leq& Ce^{-\la_0t}P\left(
\|[u_{1,0},\ta_{1,0}]\|_{H^{2s+4}}+\|u_{2,0}\|_{H^{2s}}\right),\ s\geq3.
\end{split}
\end{equation*}
Therefore \eqref{2hde} is valid. Once the estimates for $u_2$ are obtained, one can immediately show that
\begin{eqnarray}\label{f3-f4mi}
\left\{\begin{array}{rll}
&&\sum\limits_{\al_0+\al\leq s,\al_0\leq s-1}\left\|\pa_t^{\al_0}\na_x^\al \FP f_3\right\|_2\leq Ce^{-\la_0t}P\left(\|[u_{1,0},\ta_{1,0}]\|_{H^{2s+4}}+\|u_{2,0}\|_{H^{2s}}\right),\\[2mm]
&&\sum\limits_{\al_0+\al\leq s,\al_0\leq s-1}\left\|\pa_t^{\al_0}\na_x^\al \{\FI-\FP\}f_4\right\|_2\leq Ce^{-\la_0t}P\left(\|[u_{1,0},\ta_{1,0}]\|_{H^{2s+6}}+\|u_{2,0}\|_{H^{2s+2}}\right),
\end{array}\right.
\end{eqnarray}
and
\begin{equation*}\label{f3-f4mi-if}
\begin{split}
&\sum\limits_{\al_0+\al\leq s,\al_0\leq s-1}\left\|w_l\pa_t^{\al_0}\na_x^\al \FP f_3\right\|_\infty\leq Ce^{-\la_0t}P\left(\|[u_{1,0},\ta_{1,0}]\|_{H^{2s+6}}+\|u_{2,0}\|_{H^{2s+2}}\right),\\
&\sum\limits_{\al_0+\al\leq s,\al_0\leq s-1}\left\|w_l\pa_t^{\al_0}\na_x^\al \{\FI-\FP\}f_4\right\|_\infty\leq Ce^{-\la_0t}P\left(\|[u_{1,0},\ta_{1,0}]\|_{H^{2s+8}}+\|u_{2,0}\|_{H^{2s+4}}\right),
\end{split}
\end{equation*}
according to the definition \eqref{f1-5}.
We now turn to estimate $\rho_2$ and $\ta_2$.
From $\eqref{INSF2}_2$, it follows
\begin{equation}\label{rho2}
\begin{split}
\na_x\rho_2=&\frac{1}{3}\na_x(u_1\cdot u_2)-\na_x\ta_2-(u_1\na_x\cdot u_2+\na_x u_1\cdot u_2+u_1\cdot \na_x u_2)+\mu_\ast\Delta u_2\\&-\pa_tu_2
+\frac{\mu_\ast}{3}\na_x\na_x\cdot u_2
+\na_x\cdot\langle \pa_t f_2, L^{-1}A(v)\rangle
-\na_x\cdot\langle\Gamma(f_2,f_2),L^{-1}A(v)\rangle
\\&-\na_x\cdot\langle\Gamma(f_1,\{\FI-\FP\}f_3)+\Gamma(\{\FI-\FP\}f_3,f_1),L^{-1}A(v)\rangle\\
&-\left\{\pa_t(\rho_1u_1)+2\na_x\cdot(\rho u_1\otimes u_1)+\na_{x}(\rho_1\ta_1+\frac{5}{2}\ta_1^2-\frac{1}{3}\rho_1|u_1|^2)\right\}\\
&+\mu_\ast\Delta (\rho_1u_1)+\frac{\mu_\ast}{3}\na_x\na_x\cdot(\rho_1u_1)\\
\eqdef& -\na_x\ta_2+\FR_{\rho_2},
\end{split}
\end{equation}
here $\FR_{\rho_2}$ denotes the summation of all the other terms except $\na_x\ta_2$ on the right hand side of the above identity.
Inserting $\pa_t^{\al_0}\na_x^{\al}\eqref{rho2}$ $(\al_0+\al\leq s,\al_0\leq s-1)$ into $\pa_t^{\al_0}\na_x^{\al}\na_x\eqref{INSF2}_3$ and taking the inner product of the resulting equation with $\pa_t^{\al_0}\na_x^{\al}\na_x\ta_2$, one has
\begin{equation*}\label{ta2ip}
\begin{split}
\frac{5}{2}(\pa_t& \pa_t^{\al_0}\na_x^{\al}\na_x\ta_2,\pa_t^{\al_0}\na_x^{\al}\na_x\ta_2)-(\pa_t^{\al_0}\na_x^{\al}\FR_{\rho_2},\pa_t^{\al_0}\na_x^{\al}\na_x\ta_2)
\\&+\frac{5}{2}(\pa_t^{\al_0}\na_x^{\al}\na_x\na_x\cdot(u_1\ta_2),\pa_t^{\al_0}\na_x^{\al}\na_x\ta_2)+
\frac{5}{6}(\pa_t^{\al_0}\na_x^{\al}\na_x\na_x\cdot(u_1u_1\cdot u_2),\pa_t^{\al_0}\na_x^{\al}\na_x\ta_2)\\=&\ka_\ast(\pa_t^{\al_0}\na_x^{\al}\na_x\Delta\ta_2,\pa_t^{\al_0}\na_x^{\al}\na_x\ta_2)
+\ka_\ast(\pa_t^{\al_0}\na_x^{\al}\na_x\Delta(\rho_1\ta_1+\ta_1^2),\pa_t^{\al_0}\na_x^{\al}\na_x\ta_2)
\\&+\frac{2\ka_\ast}{3}(\pa_t^{\al_0}\na_x^{\al}\na_x\Delta(u_1\cdot u_2),\pa_t^{\al_0}\na_x^{\al}\na_x\ta_2)
+\frac{\ka_\ast}{3}(\pa_t^{\al_0}\na_x^{\al}\na_x\Delta(\rho_1 u_1^2),\pa_t^{\al_0}\na_x^{\al}\na_x\ta_2)
\\&-\frac{1}{2}(\pa_t^{\al_0}\na_x^{\al}\na_x\pa_t(2u_1\cdot u_2+\rho_1 u_1^2+3\rho_1\ta_1),\pa_t^{\al_0}\na_x^{\al}\na_x\ta_2)\\&
-\frac{5}{2}(\pa_t^{\al_0}\na_x^{\al}\na_x\na_x\cdot(u_1(\rho_1\ta_1+\ta_1^2)),\pa_t^{\al_0}\na_x^{\al}\na_x\ta_2)
-\frac{5}{6}(\pa_t^{\al_0}\na_x^{\al}\na_x\na_x\cdot(u_1\rho u_1^2),\pa_t^{\al_0}\na_x^{\al}\na_x\ta_2)
\\&-(\pa_t^{\al_0}\na_x^{\al}\na_x\na_x\cdot \langle L^{-1}\left\{-\pa_tf_3-v\cdot\na_x\{\FI-\FP\}f_4
+\Gamma(f_2,f_3)+\Gamma(f_3,f_2)\right\},B(v)\rangle,\pa_t^{\al_0}\na_x^{\al}\na_x\ta_2)
\\&-(\pa_t^{\al_0}\na_x^{\al}\na_x\na_x\cdot \langle L^{-1}\left\{\Gamma(f_1,\{\FI-\FP\}f_4)+\Gamma(\{\FI-\FP\}f_4,f_1)\right\},B(v)\rangle
,\pa_t^{\al_0}\na_x^{\al}\na_x\ta_2).
\end{split}
\end{equation*}
By integration by parts and using Cauchy-Schwartz's inequality with $\eta>0$, we deduce
\begin{equation*}
\begin{split}
\frac{d}{dt}\|&\pa_t^{\al_0}\na_x^{\al}\na_x\ta_2\|_2^2+\la_3\|\pa_t^{\al_0}\na_x^{\al}\na^2_x\ta_2\|_2^2\\
\leq& C_\eta\sum\limits_{\al_0+\al\leq s+2,\al_0\leq s+1}\|\pa_t^{\al_0}\na_x^{\al}[\rho_1,u_1,\ta_1,u_2]\|^2_2
\\&+C_\eta\sum\limits_{\al_0+\al\leq s+2,\al_0\leq s+1}\left\|\pa_t^{\al_0}\na_x^{\al}[f_1,f_2,f_3,\{\FI-\FP\}f_4]\right\|^2_2
+\eta\|\pa_t^{\al_0}\na_x^{\al}\na_x\ta_2\|_{H^1}^2.
\end{split}
\end{equation*}
Poincar$\acute{e}$'s inequality further yields
\begin{equation}\label{ta2ip2}
\begin{split}
\|\pa_t^{\al_0}\na_x^{\al}\na_x\ta_2\|_2^2
\leq& C\sum\limits_{\al_0+\al\leq s+2,\al_0\leq s+1}e^{-\la_3t}\int_0^te^{\la_3s}\|\pa_t^{\al_0}\na_x^{\al}[\rho_1,u_1,\ta_1,u_2]\|^2_2ds
\\&+C\sum\limits_{\al_0+\al\leq s+2,\al_0\leq s+1}e^{-\la_3t}\int_0^te^{\la_3s}\left\|\pa_t^{\al_0}\na_x^{\al}[f_1,f_2,f_3,\{\FI-\FP\}f_4]\right\|^2_2ds.
\end{split}
\end{equation}
Invoking \eqref{meanm}, $\eqref{INSF2}_3$  and $\eqref{INSF2}_4$, one has
\begin{equation}\label{ta2con}
\int_{\T^3}(3\ta_{2}+2u_{1}\cdot u_{2}+\rho_{1}|u_{1}|^2+3\rho_{1}\ta_{1})dx=0,\ \int_{\T^3}\rho_2dx=0.
\end{equation}
Combing now \eqref{f1-f2l2}, \eqref{2hde}, \eqref{f3-f4mi}, \eqref{ta2ip2} and \eqref{ta2con} and taking $\la_3<\la_0$, we deduce
\begin{equation}\label{rho2dec}
\begin{split}
\sum\limits_{\al_0+\al\leq s,\al_0\leq s-1}\left\|\pa_t^{\al_0}\na_x^\al [\rho_2,\ta_2]\right\|_2
\leq& Ce^{-\la_0t}P\left(
\|[u_{1,0},\ta_{1,0}]\|_{H^{2s+8}}+\|[u_{2,0},\ta_{2,0}]\|_{H^{2s+6}}\right)
\\&+C\sum\limits_{\al_0+\al\leq s,\al_0\leq s-1}\left\|\pa_t^{\al_0}\na_x^\al \na_x(2u_1\cdot u_2+\rho_1 u_1^2+3\rho_1\ta_1)\right\|_2\\&+C\sum\limits_{\al_0+\al\leq s,\al_0\leq s-1}\left\|\pa_t^{\al_0}\na_x^\al (2u_1\cdot u_2+\rho_1 u_1^2+3\rho_1\ta_1)\right\|_2\\
\leq& Ce^{-\la_0t}P\left(
\|[u_{1,0},\ta_{1,0}]\|_{H^{2s+8}}+\|[u_{2,0},\ta_{2,0}]\|_{H^{2s+6}}\right).
\end{split}
\end{equation}
We now conclude from \eqref{f1-5}, \eqref{f1-f2l2}, \eqref{f3mi2}, \eqref{2hde} and \eqref{rho2dec}
as well as Lemma \ref{es.nop} that
\begin{equation}\label{f4l2ma}
\begin{split}
\sum\limits_{\al_0+\al\leq s,\al_0\leq s-1}\left\|\pa_t^{\al_0}\na_x^\al\FP f_4\right\|_2\leq Ce^{-\la_0t}P\left(
\|[u_{1,0},\ta_{1,0}]\|_{H^{2s+8}}+\|[u_{2,0},\ta_{2,0}]\|_{H^{2s+6}}\right),
\end{split}
\end{equation}
and
\begin{equation}\label{f5l2mi}
\begin{split}
\sum\limits_{\al_0+\al\leq s,\al_0\leq s-1}\left\|\pa_t^{\al_0}\na_x^\al\{\FI-\FP\} f_5\right\|_2\leq Ce^{-\la_0t}P\left(
\|[u_{1,0},\ta_{1,0}]\|_{H^{2s+10}}+\|[u_{2,0},\ta_{2,0}]\|_{H^{2s+8}}\right).
\end{split}
\end{equation}
As to $\FP f_5$, it suffices to determine $u_3$. Since the expansion \eqref{oe.ep} is truncated at $f_6$, we can assume $\FP_0 u_3=0$.
Moreover, \eqref{R} and \eqref{meanm} imply
$
\int_{\R^3}(u_3+\rho_2u_1+\rho_1u_2)dx=0.
$
Those ensure us to obtain from
 $\eqref{INSF2}_4$ that
\begin{equation}\label{u3}
\begin{split}
\sum\limits_{\al_0+\al\leq s,\al_0\leq s-1}\left\|\pa_t^{\al_0}\na_x^\al u_3\right\|_2\leq& C\sum\limits_{\al_0+\al\leq s,\al_0\leq s-1}\left\|\pa_t^{\al_0}\na_x^\al[\pa_t\rho_1,\rho_1,\rho_2,u_1,u_2]\right\|_2\\
\leq& Ce^{-\la_0t}P\left(
\|[u_{1,0},\ta_{1,0}]\|_{H^{2s+2}}+\|[u_{2,0},\ta_{2,0}]\|_{H^{2s}}\right).
\end{split}
\end{equation}
Therefore one deduces from \eqref{f1-f2l2}, \eqref{f3mi2}, \eqref{f3-f4mi}, \eqref{f4l2ma}, \eqref{f5l2mi}
and \eqref{u3}
\begin{equation*}\label{f5l2}
\begin{split}
\sum\limits_{\al_0+\al\leq s,\al_0\leq s-1}\left\|\pa_t^{\al_0}\na_x^\al f_5\right\|_2\leq Ce^{-\la_0t}P\left(
\|[u_{1,0},\ta_{1,0}]\|_{H^{2s+10}}+\|[u_{2,0},\ta_{2,0}]\|_{H^{2s+8}}\right),
\end{split}
\end{equation*}
and
\begin{equation*}\label{f6l2}
\begin{split}
\sum\limits_{\al_0+\al\leq s,\al_0\leq s-1}\left\|\pa_t^{\al_0}\na_x^\al f_6\right\|_2\leq Ce^{-\la_0t}P\left(
\|[u_{1,0},\ta_{1,0}]\|_{H^{2s+12}}+\|u_{2,0}\|_{H^{2s+10}}\right).
\end{split}
\end{equation*}
Finally,
using Lemma \ref{gadec} and the fact that $L^{-1}$ preserves the decay in $v$, we see that \eqref{ssol} is also valid.
This ends the proof of Proposition \ref{ss}.

\end{proof}

\section{$L^2-$ theory}\label{l2leq}
In this section, we will study the solutions of the linear equation of the remainder which satisfies \eqref{R} in $L^2$ setting.
The main purpose of this section is to prove the following:

\begin{proposition}
\label{dlinearl2}Assume $ g_1,g_2 \in L^{2}(\mathbb{R}%
_{+} \times \T^3 \times \mathbb{R}^{3})$ and for all $t>0$,
\begin{equation}
\int_{\T^3 \times \mathbb{R}^{3}}g_1(t,x,v)[1,v,v^2]\sqrt{\mu }\mathrm{d} v \mathrm{d%
} x=0 .
\label{dlinearcondition}
\end{equation}
Then, for $g=\eps g_1+g_2$ and for any sufficiently small $\eps$, there exists a unique solution to the
problem
\begin{eqnarray}\label{dlinear}
\left\{\begin{array}{rll}
&\eps\partial _{t}f+v\cdot \nabla _{x}f+\eps^{-1}Lf=g,\ x\in \T^3,\ v\in\R^3, \\[2mm]
&f(0,x,v)=f_{0}(x,v),\ x\in \T^3,\ v\in\R^3,
\end{array}\right.
\end{eqnarray}%
such that
\begin{equation}
\int_{\T^3 \times \mathbb{R}^{3}}f(t,x,v)[1,v,v^2]\sqrt{\mu }\mathrm{d} x\mathrm{d}
v=0, \ \ \ \text{for all} \ t\geq0.  \label{dlinearcondition1}
\end{equation}
Moreover, there is $0<\lambda \ll 1$ such that for $t\geq0$,
\begin{equation}\label{completes_dyn}
\begin{split}
\| e^{\lambda t}f(t)\|_2^2&+ \eps^{-2}\int_0^t \| e^{\lambda \tau}\{\mathbf{I}-%
\mathbf{P}\} f (\tau)\|_\nu^2 \mathrm{d} \tau + \int^{t}_{0} \| e^{\lambda
\tau}\mathbf{P} f (\tau)\|_{2}^{2} \mathrm{d} \tau    \\
\lesssim& \|f_0\|_2^2+\int_0^t \| \nu^{- \frac{1}{2}}e^{\lambda \tau} \{{
\mathbf{I} - \mathbf{P}}\} g  \|_{2}^{2} + \eps^{-2} \int_0^t \| e^{\lambda
\tau} {\ \mathbf{P} g } \|_{2}^{2} .
\end{split}
\end{equation}
\end{proposition}
To prove Proposition \ref{dlinearl2}, let us first show that the macroscopic part of the solution of \eqref{dlinear}
can be dominated by its microscopic part, for results in this direction, we have
\begin{lemma}
\label{dabc}Assume $g=\eps g_1+g_2$ with $g_1$ satisfying \eqref{dlinearcondition} and $f$
satisfies \eqref{dlinear} and \eqref{dlinearcondition1}. Then there exists
a function $G(t)$ such that, for all $t\geq 0$, $G(t)\lesssim\eps
\|f(t)\|_{2}^{2}$ and
\begin{equation*}\label{mm}
\int_{0}^{t}\|\mathbf{P}f(\tau)\|_{\nu }^{2}d\tau \lesssim
G(t)-G(0)+\int_{0}^{t}\|\nu^{-1/2}g(\tau) \|_{2}^{2}d\tau
+ \eps^{-2}\int_{0}^{t}\|\{\mathbf{I}-\mathbf{P}
\}f(\tau)\|_{\nu }^{2}d\tau.
\end{equation*}
\end{lemma}
\begin{proof}
The proof is the same as
Lemma 3.9 in \cite[pp.45]{EGKM-15} or Lemma 6.1 in \cite[pp.656]{Guo-2006} with some trivial modification.
\end{proof}
We are now in a position to complete
\begin{proof}
[The proof of Proposition \ref{dlinearl2}]
Notice that \eqref{dlinear} is a linear problem, whose global existence is easy to be seen, in what follows, we only prove
\eqref{completes_dyn}. Let $y(t)= e^{\lambda t}f(t)$ with $\la>0$. We multiply \eqref{dlinear}
by $e^{\lambda t}$, so that $y$ satisfies
\begin{equation}
\partial _{t}y+\eps^{-1}v\cdot \nabla _{x}y+\eps^{-2}Ly=\lambda y+e^{\lambda t}\eps^{-1}g,\ y(0,x,v)=f_0(x,v),\ x\in \T^3,\ v\in\R^3.  \label{lineary}
\end{equation}%
Taking the inner product of \eqref{lineary} with $y$ over $\T^3\times\R^3$ and integrating the resulting equation with respect to time, one has
\begin{equation}\label{L2ip}
\begin{split}
\frac 1 2\| y(t)\| _{2}^{2}& +\eps^{-2}\int_{0}^{t}\| \{\mathbf{I}-\mathbf{P}
\}y(s)\| _{\nu }^{2}ds\\
\leq & (\lambda+\eta) \int_{0}^{t}\| y(s)\| _{2}^{2}+\| y(0)\| _{2}^{2}+\int_{0}^{t}e^{\lambda s}\|
\nu^{- \frac{1}{2}} \{\mathbf{I} - \mathbf{P }\}g \| _{2}^{2}ds+\eps^{-2}C_\eta\int_{0}^{t}e^{\lambda s}\| g \| _{2}^{2}ds.
\end{split}%
\end{equation}
Applying Lemma \ref{dabc} to (\ref{lineary}), we deduce
\begin{equation}\label{mm2}
\int_{0}^{t}\| \mathbf{P}y(s)\| _{\nu }^{2}\mathrm{d} s \lesssim G(t)-G(0)+%
\eps^{-2}\int_{0}^{t}\| \{\mathbf{I}-\mathbf{P}\}y(s)\| _{\nu }^{2}\mathrm{d}
s+\int_{0}^{t}e^{\lambda s}\| g\| _{2}^{2}\mathrm{d} s  +\lambda \int_{0}^{t}\| y\| _{2}^{2}\mathrm{d} s,
\end{equation}
where $G(t)\lesssim \eps  \| y(t)\| _{2}^{2}$.
\eqref{completes_dyn} thereby follows from a linear combination of \eqref{L2ip} and \eqref{mm2}. This finishes the proof of Proposition
\ref{dlinearl2}.

\end{proof}

\section{$L^\infty-$theory}\label{lifleq}
This section is dedicated to obtaining the $L^\infty-$ estimates of the solution to the linear equation \eqref{dlinear}. More precisely, we are going to prove the following:
\begin{proposition}
\label{point_dyn} Assume $f$ satisfies
\begin{eqnarray}  \label{linear_K}
\left\{
\begin{array}{rll}
&&\left[  \partial_{t} + \eps^{-1}v \cdot \nabla_{x} +
\eps^{-2} \nu(v) \right] f  = \eps^{-2} K f+ \eps^{-1}g,
\\[2mm]
&&f(0,x,v)=f_{0}(x,v), \ x\in\T^3,\ v\in\R^3.
\end{array}\right.
\end{eqnarray}
Then, for $l\geq0$, there exists $\la>0$ such that
\begin{equation}  \label{point2}
\begin{split}
\| \eps^{\frac 3 2} w_{l} f(t) \|_{\infty} \lesssim e^{-\la t}\| \eps^{\frac 3 2} w_l f_{0}
\|_{\infty} + \eps^{\frac 5 2} e^{-\la t}\sup_{0 \leq s \leq t} \| e^{\la s}\nu^{-1} w_l g(s)\|_{\infty}
 +  e^{-\la t}\sup_{0 \leq s \leq t}\| e^{\la s}f(s)\|_{2} .
\end{split}%
\end{equation}
\end{proposition}
\begin{proof}
Notice that the equations of the characteristics for \eqref{linear_K} are
\begin{equation*}\label{ch}
\frac{d X_\eps(s)}{ds}=\frac{V(s)}{\eps},\ \frac{d V(s)}{ds}=0,
\end{equation*}
with initial data $[X(t;t,x,v),V(t;t,x,v)]=[x,v].$ By this, we write $X_\eps(s)=X_\eps(s;t,x,v)=x+\frac{s-t}{\eps}v$
and $V(s)=V(s;t,x,v)=v.$ Let $h=w_lf$, we then get from Duhamel's principle that
\begin{equation*}\label{hexp}
\begin{split}
\eps^{3/2}h(t,x,v)=&\eps^{3/2}e^{-\frac{\nu(v)t}{\eps^2}}h_0(x-\frac{v}{\eps}t,v)
\\&+\int_0^te^{-\frac{\nu(v)(t-s)}{\eps^2}}\left[\eps^{-1/2}K_w h+\eps^{1/2}w_lg\right](s,x+\frac{(s-t)v}{\eps},v)ds,
\end{split}
\end{equation*}
where $K_w(\cdot)=w_l K(\frac{\cdot}{w_l})$ and $h_0(x,v)=f_0(x,v)w_l$.
Direct calculation yields
\begin{equation}\label{hifip}
\begin{split}
|\eps^{3/2}h(t,x,v)|\leq &\eps^{3/2}e^{-\frac{\nu_0t}{\eps^2}}\|h_0\|_{\infty}+C\eps^{5/2}e^{-\la t}\sup\limits_{0\leq s\leq t}\left\{e^{\la s}\|\nu^{-1}w_lg(s)\|_\infty\right\}
\\&+\underbrace{\int_0^te^{-\frac{\nu(v)(t-s)}{\eps^2}}\eps^{-1/2}|K_w h(s,x+\frac{(s-t)v}{\eps},v)|}_{J_1}ds.
\end{split}
\end{equation}
Here, $\nu_0$ is a constant and satisfies $0<\nu_0\leq \nu(v)$ and $\la\leq \frac{\nu_0 }{2\eps^2}$.
We further iterate this formula to evaluate $J_1$ as
\begin{eqnarray}\label{J1}
\begin{split}
J_1\leq& \int_0^te^{-\frac{\nu(v)(t-s)}{\eps^2}}\frac{1}{\eps^2}\int_{\R^3}k_w(v,v')|\eps^{3/2}h(s,x+\frac{(s-t)v}{\eps},v')|dv'ds\\
\leq &\int_0^te^{-\frac{\nu(v)(t-s)}{\eps^2}}\frac{1}{\eps^2}\left\{\eps^{3/2}e^{-\frac{\nu_0s}{\eps^2}}\|h_0\|_{\infty}
+C\eps^{5/2}e^{-\la\tau}\sup\limits_{0\leq \tau\leq s}\left\{e^{\la\tau}\|\nu^{-1}w_lg(\tau)\|_\infty\right\}\right\}ds
\int_{\R^3}k_w(v,v')dv'\\
&+\int_0^te^{-\frac{\nu(v)(t-s)}{\eps^2}}\int_0^se^{-\frac{\nu(v')(s-\tau)}{\eps^2}}\eps^{-5/2}\int_{\R^3}k_w(v,v')k_w(v',v'')
\\&\qquad\times|h(\tau,x+\frac{(s-t)v}{\eps}+\frac{(s-\tau)v'}{\eps},v'')|dv'dv''d\tau ds\\
\leq &C\eps^{3/2}e^{-\frac{\nu_0t}{2\eps^2}}\|h_0\|_{\infty}
+C\eps^{5/2}e^{-\frac{\nu_0t}{2\eps^2}}\sup\limits_{0\leq \tau\leq t}
\left\{e^{\frac{\nu_0 \tau}{2\eps^2}}\|\nu^{-1}w_lg(\tau)\|_\infty\right\}
\\&+\underbrace{\int_0^te^{-\frac{\nu(v)(t-s)}{\eps^2}}\int_0^se^{-\frac{\nu(v')(s-\tau)}{\eps^2}}\eps^{-5/2}
\int_{\R^6}k_w(v,v')k_w(v',v'')
|h(\tau,X_\eps(\tau;X_\eps(s),v'),v'')|dv'dv''d\tau ds}_{J_2},
\end{split}
\end{eqnarray}
where $X_\eps(\tau;X_\eps(s),v')=x+\frac{(s-t)v}{\eps}+\frac{(s-\tau)v'}{\eps},$ and $k_w(v,v')=w_l(v)k(v,v')\frac{1}{w_l(v')}$ with $k(v,v')$ given by \eqref{Kk}.
To compute $J_2$, we first split it into
$$
J_2=\int_0^te^{-\frac{\nu(v)(t-s)}{\eps^2}}\int_{s-\ka\eps^2}^se^{-\frac{\nu(v')(s-\tau)}{\eps^2}}\cdots d\tau ds
+\int_0^te^{-\frac{\nu(v)(t-s)}{\eps^2}}\int^{s-\ka\eps^2}_0e^{-\frac{\nu(v')(s-\tau)}{\eps^2}}\cdots d\tau ds\eqdef J_{2,1}+J_{2,2},
$$
where $\ka$ is positive and sufficiently small.

Let us now turn to compute $J_{2,1}$ and $J_{2,2}$.
For $J_{2,1}$,
it is straightforward to see that
\begin{equation}\label{J21}
\begin{split}
J_{2,1}\leq& \int_{0}^{t}\int_{s-\ka\eps^2
}^{s}C_{K}e^{-\frac{\nu _{0}(t-s)}{\eps^2}}\eps^{-5/2}\|h(\tau)\|_{\infty} d\tau ds \\
\leq &C_{K}e^{\frac{-\nu _{0}t}{2\eps^2}}\int_{0}^{t}
\int_{s-\ka\eps^2}^{s}e^{\frac{-\nu _{0}(t-s)}{2\eps^2}}\left\{e^{\frac{\nu _{0}s}{2\eps^2}\eps^{-5/2}
}\|h(\tau)\|_{\infty}\right\}d\tau ds \\
\leq &C_{K}e^{\frac{-\nu _{0}t}{2\eps^2}}\eps^{-5/2}\sup\limits_{0\leq \tau\leq t}\left\{e^{\frac{\nu_0 \tau}{2\eps^2}}\|h(\tau)\|_{\infty}\right\}\times \int_{0}^{t}\int_{s-\ka\eps^2
}^{s}e^{\frac{-\nu _{0}(t-s)}{2\eps^2}}d\tau ds   \\
\leq &C_{K}e^{\frac{-\nu _{0}t}{2\eps^2}}\eps^{-1/2}\sup\limits_{0\leq \tau\leq t}\left\{e^{\frac{\nu_0 \tau}{2\eps^2}}\|h(\tau)\|_{\infty}\right\}\times \ka\eps^2 \int_{0}^{t}e^{\frac{-\nu
_{0}(t-s)}{2\eps^2}}\eps^{-2}ds  \\
\leq &C_{K}\ka e^{\frac{-\nu _{0}t}{2\eps^2}}\eps^{3/2}\sup\limits_{0\leq \tau\leq t}\left\{e^{\frac{\nu_0 \tau}{2\eps^2}}\|h(\tau)\|_{\infty}\right\}.
\end{split}
\end{equation}
As to $J_{2,2}$, the estimates are divided into following three cases:

\noindent{\it Case 1:} $|v|\geq N$ with $N$ being positive and large. In this case, Lemma \ref{es.k} implies
$$
\int_{\R^6}k_w(v,v')k_w(v',v'')dv'dv''\leq C(1+|v|)^{-1}\leq \frac{C}{1+N}.
$$
Therefore
\begin{equation}\label{J22}
\begin{split}
J_{2,2}\leq& \frac{C}{1+N}e^{-\frac{\nu_0t}{2\eps^2}}\eps^{3/2}\sup\limits_{0\leq \tau\leq t}\left\{e^{\frac{\nu_0 \tau}{2\eps^2}}\|h(\tau)\|_{\infty}\right\}
\int_0^te^{-\frac{\nu_0(t-s)}{2\eps^2}}\eps^{-2}\int_0^se^{-\frac{\nu_0(s-\tau)}{2\eps^2}}\eps^{-2}d\tau ds
\\ \leq& \frac{C}{1+N}e^{-\frac{\nu_0t}{2\eps^2}}\eps^{3/2}\sup\limits_{0\leq \tau\leq t}\left\{e^{\frac{\nu_0 \tau}{2\eps^2}}\|h(\tau)\|_{\infty}\right\}.
\end{split}
\end{equation}

\noindent{\it Case 2:} $|v|< N$, $|v'|\geq2N$, or $|v'|\leq 2N$, $|v''|\geq 3N.$ Observe that we have either $|v'-v|\geq N$
or $|v''-v'|\geq N$ and either one of the following holds accordingly for $\vps>0$
\begin{equation*}
|k_{w}(v,v^{\prime })|\leq Ce^{-\frac{\varepsilon }{8}N^{2}}|k_{w}(v,v^{\prime })e^{\frac{\varepsilon }{8}|v-v^{\prime }|^{2}}|,\text{ \
\ \ \ \ }|k_{w}(v^{\prime },v^{\prime \prime })|\leq Ce^{-\frac{%
\varepsilon }{8}N^{2}}|k_{w}(v^{\prime },v^{\prime \prime })e^{%
\frac{\varepsilon }{8}|v^{\prime }-v^{\prime \prime }|^{2}}|,  \label{kwe}
\end{equation*}
from which and Lemma \ref{es.k}, it follows that
\begin{equation}\label{J221}
\begin{split}
J_{2,2}\leq& \int_{0}^{t}e^{-\frac{\nu_0(t-s)}{\eps^2}}\int_{0}^{s}e^{-\frac{\nu_0(s-\tau)}{\eps^2}}\eps^{-5/2}\left\{ \int_{|v|\leq N,|v^{\prime }|\geq 2N,\text{ \
\ }}+\int_{|v^{\prime }|\leq 2N,|v^{\prime \prime }|\geq 3N}\right\}
\\ \leq&Ce^{-\frac{\varepsilon }{8}N^{2}}e^{-\frac{\nu_0t}{2\eps^2}}\eps^{3/2}\sup\limits_{0\leq \tau\leq t}\left\{e^{\frac{\nu_0 \tau}{2\eps^2}}\|h(\tau)\|_{\infty}\right\}
\int_0^te^{-\frac{\nu_0(t-s)}{2\eps^2}}\eps^{-2}\int_0^se^{-\frac{\nu_0(s-\tau)}{2\eps^2}}\eps^{-2}d\tau ds
\\ \leq& Ce^{-\frac{\varepsilon }{8}N^{2}}e^{-\frac{\nu_0t}{2\eps^2}}\eps^{3/2}\sup\limits_{0\leq \tau\leq t}\left\{e^{\frac{\nu_0 \tau}{2\eps^2}}\|h(\tau)\|_{\infty}\right\}.
\end{split}
\end{equation}
{\it Case 3:}
$|v|\leq N,$ $|v^{\prime
}|\leq 2N,|v^{\prime \prime }|\leq 3N.$ 
In this situation,
the velocity domain is bounded and most importantly there is a lower bound $s-\tau>\ka \eps^2$, which ensures us to convert the $L^\infty-$norm
into $L^2-$norm.
To do so, for any large $N>0$,
we first choose a number $m(N)$ to define
\begin{equation*}
k_{w,m}(p,v')\equiv \mathbf{1}
_{|p-v^{\prime }|\geq \frac{1}{m},|v^{\prime}|\leq m}k_{w}(p,v'),
\label{km}
\end{equation*}%
such that $\sup_{p}\int_{\R^{3}}|k
_{w,m}(p,v^{\prime})
-k_{w}(p ,v^{\prime})|dv^{\prime}\leq
\frac{1}{N}.$ We then split
\begin{equation*}
\begin{split}
k_w(v,v')
k_{w}(v^{\prime },v^{\prime \prime})=&\{k_{w}(v,v^{\prime})
-k_{w,m}(v,v^{\prime})\}k_{w}(v^{\prime },v^{\prime \prime })
\\&+\{k_{w}(v',v'')
-k_{w,m}(v',v'')\}k_{w,m}(v,v^{\prime})
+k_{w,m}(v,v^{\prime})k_{w,m}(v^{\prime },v^{\prime \prime }),
\end{split}
\end{equation*}
one can use such an approximation to bound the above $J_{2,2}$ by
\begin{eqnarray}
&&\frac{Ce^{-\frac{\nu _{0}t}{2\eps^2}}}{N}\eps^{3/2}\sup\limits_{0\leq \tau\leq t}\left\{e^{\frac{\nu_0 \tau}{2\eps^2}}\|h(\tau)\|_{\infty}\right\} \left\{ \sup_{|v^{\prime }|\leq 2N}\int |%
k_{w}(v^{\prime },v^{\prime \prime })|dv^{\prime \prime
}+\sup_{|v|\leq 2N}\int |k_{w,m}(v,v^{\prime })|dv^{\prime }\right\}
 \notag\\
&&+C\int_{0}^{t}\int_{0}^{s-\ka\eps^2}\int_{|v'|\leq 2N,|v''|\leq 3N}
e^{-\frac{\nu(v)(t-s)}{\eps^2}}e^{-\frac{\nu(v')(s-\tau)}{\eps^2}}\label{J222}\\
&&\quad\times\eps^{-5/2}k_{w,m}(v,v^{\prime })k_{w,m}(v^{\prime },v^{\prime \prime
})|h(\tau,X_\eps(\tau;X_\eps(s),v'),v'')|dv'dv''.  \notag
\end{eqnarray}%
Next, by a change of variable $y=X_\eps(\tau;X_\eps(s),v')=x+\frac{(s-t)v}{\eps}+\frac{(s-\tau)v'}{\eps},$ and for $s-\tau\geq \ka\eps^2 ,$ $\frac{dy}{dv^{\prime }}\geq
\ka ^{3}\eps^3,$ we can further control the
last term in \eqref{J222} by:%
\begin{equation} \label{fif-2}
\begin{split}
\frac{C_{N}}{\eps ^{3/2}}&\int_{0}^{t}\int_{0}^{s-\ka\eps^2}
e^{-\frac{\nu_0(t-s)}{\eps^2}}e^{-\frac{\nu_0(s-\tau)}{\eps^2}}\eps^{-5/2}\int_{|v''|\leq 3N}
\left\{ \int_{|y-X_\eps(s)|\leq \frac{2(s-\tau)N}{\eps} }|h(\tau,y,v^{\prime \prime
})|^{2}dy\right\} ^{1/2}dv^{\prime \prime }d\tau ds  \\
\leq &\frac{C_{N}}{\eps ^{3/2}}\int_{0}^{t}\int_{0}^{s-\ka\eps^2}
e^{-\frac{\nu_0(t-s)}{\eps^2}}e^{-\frac{\nu_0(s-\tau)}{\eps^2}}\eps^{-5/2}\left[(\frac{s-\tau}{\eps})^{3/2}+1\right]
\\
&\times\left\{\int_{|v''|\leq 3N} \int_{\T^3 }|h(\tau,y,v^{\prime \prime
})|^{2}dydv^{\prime \prime }\right\} ^{1/2}d\tau ds   \\
\leq &\frac{C_{N}}{\eps ^{3/2}}\int_{0}^{t}\int_{0}^{s-\ka\eps^2}
e^{-\frac{\nu_0(t-s)}{\eps^2}}e^{-\frac{\nu_0(s-\tau)}{\eps^2}}\eps^{-5/2}\left[(\frac{s-\tau}{\eps})^{3/2}+1\right]
\\
&\times\left\{\int_{|v''|\leq 3N} \int_{\T^3 }|f(\tau,y,v^{\prime \prime
})|^{2}dydv^{\prime \prime }\right\} ^{1/2}d\tau ds  \\
\leq &C_Ne^{-\lambda t}\sup_{s\geq
0}\left\{e^{\lambda s}\|f(s)\|_2\right\}\int_{0}^{t}\int_{0}^{s-\ka\eps^2}
e^{-\frac{\nu_0(t-s)}{2\eps^2}}e^{-\frac{\nu_0(s-\tau)}{2\eps^2}}e^{-\la\tau}\eps^{-4}\left[(\frac{s-\tau}{\eps})^{3/2}+1\right]d\tau ds\\
\leq& C_Ne^{-\lambda t}\sup_{s\geq
0}\left\{e^{\lambda s}\|f(s)\|_2\right\},
\end{split}
\end{equation}%
where we have used the fact that $\la\leq\frac{\nu _{0}}{2\eps^2}.$

Inserting \eqref{J1}, \eqref{J21}, \eqref{J22}, \eqref{J221}, \eqref{J222} and \eqref{fif-2} into \eqref{hifip}, one can see that \eqref{point2}
is true, which concludes the proof of Proposition \ref{point_dyn}.

\end{proof}

\section{Global existence and time decay}\label{proof}
In this final section, we will prove the global existence and exponential time decay of the solutions to the equation \eqref{R} in $L^2\cap L^\infty-$framework. That is we intend to complete
\begin{proof}[The proof of Theorem \ref{mre}]
Recall the Cauchy problem for the linearized equation \eqref{dlinear} or \eqref{linear_K}, to prove the global existence of
\eqref{R} with $R(0,x,v)=R_0(x,v)$, let us first design
the following iteration sequence
\begin{eqnarray}\label{Rit}
\left\{\begin{array}{rll}
\begin{split}
&\eps\pa_tR^{\ell+1}+v\cdot\na_xR^{\ell+1}+\frac{1}{\eps}LR^{\ell+1}=g(R^{\ell}),\\
&R^{\ell+1}(0,x,v)=R_0(x,v),\ R^{0}=R_0(x,v), \ x\in\T^3,\ v\in \R^3,
\end{split}
\end{array}\right.
\end{eqnarray}
where $g(R^{\ell})$ is defined by
\begin{equation}\label{Rit}
\begin{split}
g(R^{\ell})
=&\left\{\Gamma(f_1,R^{\ell})+\Gamma(R^{\ell},f_1)\right\}
+\eps\left\{\Gamma(f_2,R^{\ell})+\Gamma(R^{\ell},f_2)\right\}\\
&+\eps^2\left\{\Gamma(f_3,R^{\ell})+\Gamma(R^{\ell},f_3)\right\}
+\eps^3\left\{\Gamma(f_4,R^{\ell})+\Gamma(R^{\ell},f_4)\right\}
\\
&+\eps^4\left\{\Gamma(f_5,R^{\ell})+\Gamma(R^{\ell},f_5)\right\}
+\eps^5\left\{\Gamma(f_6,R^{\ell})+\Gamma(R^{\ell},f_6)\right\}
\\&+\eps^{4-\be}\Gamma(R^{\ell},R^{\ell})
-\eps^{1+\beta}\left\{\pa_tf_5+v\cdot\na_xf_6\right\}-\eps^{2+\beta} \pa_t f_6.
\end{split}
\end{equation}
Clearly, \eqref{Rit} satisfies the conditions listed in Proposition \ref{dlinearl2}
with $g=g(R^{\ell})$.

It is important to note that the iteration scheme \eqref{Rit} does not provide us the positivity of the solution of the original equation \eqref{BE}, however it coincides with the linearized equation \eqref{dlinear} so that Propositions \ref{dlinearl2}
and \ref{point_dyn} can be directly used. 
Let us now define the following energy functional
$$
\CE(f)(t)=e^{2\la t}\eps^{3}\|w_lf(t)\|^2_{\infty}+e^{2\la t}\|f(t)\|_2^2,
$$
and dissipation rate
$$
\mathcal {D}(f)(t)=\eps^{-2}e^{2\la t}\|\{\FI-\FP\}f(t)\|_{\nu}^2+e^{2\la t}\|\FP f(t)\|_2^2.
$$
For later use, we also define a Banach space
$$
\FX_\de(t)=\left\{f~|~\sup\limits_{0\leq s\leq t}\CE(f)(s)+\int_0^t\mathcal {D}(f)(s)ds<\de,\ \ \de>0\right\},
$$
endowed with the norm
$$
\|f\|_{\FX_\de}=\sup\limits_{0\leq s\leq t}\CE(f)(s)+\int_0^t\mathcal {D}(f)(s)ds.
$$
We now show that $R^{\ell+1}\in\FX_\de$ if $R^{\ell}\in\FX_\de$. For this,
on the one hand, we know from \eqref{completes_dyn} 
and \eqref{point2} 
with $f=R^{\ell+1}$ and
$g=g (R^{\ell })$ that \eqref{Rit} admits a unique solution $R^{\ell+1}$ satisfying
\begin{equation}\label{pro24}
\begin{split}
\sup\limits_{0\leq s\leq t}&\CE(R^{\ell+1})(s)+\int_0^t\mathcal {D}(R^{\ell+1})(s)ds\\
\leq& C\CE(f)(0)+ C\eps^{5}\sup\limits_{0\leq s\leq t}e^{2\la s}\left\|\nu^{-1}w_lg (R^{\ell })(s)\right\|^2_\infty
\\
&+C\int_{0}^te^{2\la s}\left\|\nu^{-1/2}\{\FI-\FP\}g (R^{\ell })(s)\right\|_2^2ds
+C\eps^{-2}\int_{0}^{t}e^{2\lambda  s}\left\| \nu^{-1/2}\FP g (R^{\ell })(s)\right\|
_{2}^{2}ds.
\end{split}
\end{equation}
On the another hand, thanks to Lemmas \ref{es.nop} and \ref{gadec} as well as Proposition \ref{ss}, it follows for $\la_0>\la>0$ and $l>3/2$
\begin{equation}\label{es.nopl1}
\begin{split}
\int_{0}^t&e^{2\lambda s}\left\|\nu^{-1/2}\{\FI-\FP\}g (R^{\ell })(s)\right\|_2^2ds\\
\leq& C\sup\limits_{0\leq s\leq t}\|w_l\nu f_1(s)\|^2_{\infty}\int_{0}^te^{2\lambda s}\|R^\ell(s)\|^2_{\nu}ds
+C\sum\limits_{i=2}^6\eps^{2(i-1)}\|w_l \nu f_i(s)\|^2_{\infty}\int_{0}^te^{2\lambda s}\|R^\ell(s)\|^2_{\nu}ds\\
&+C\eps^{8-2\be}\sup\limits_{0\leq s\leq t}\|w_lR^\ell\|^2_{\infty}\int_{0}^te^{2\lambda s}\|R^\ell(s)\|^2_{\nu}ds
+\eps^{2+2\be}\int_{0}^te^{2\lambda s}\left\|\{\FI-\FP\}\left\{\pa_t f_5+v\cdot \na_xf_6\right\}\right\|_2^2ds
\\&+\eps^{4+2\be}\int_{0}^te^{2\lambda s}\|\pa_t f_6\|_2^2ds\\
\leq& C\left\{\sup\limits_{0\leq s\leq t}\CE(R^\ell)(s)+\vps_0^2+\eps^2P^2\left(
\|[u_{1,0},\ta_{1,0}]\|_{H^{14}}+\|[u_{2,0},\ta_{2,0}]\|_{H^{12}}\right)\right\}\int_{0}^t\mathcal {D}(R^\ell)(s)ds
\\&+\eps^{2+2\be}\int_{0}^te^{2\lambda s}e^{-2\la_0 s}P^2\left(
\|[u_{1,0},\ta_{1,0}]\|_{H^{14}}+\|[u_{2,0},\ta_{2,0}]\|_{H^{12}}\right)ds,
\end{split}
\end{equation}
\begin{equation}\label{es.nopl2}
\begin{split}
\eps^{-2}&\int_{0}^{t}e^{2\lambda s}\Vert \nu^{-1/2}\FP g (R^{\ell })(s)\Vert
_{2}^{2}ds\\ \leq& C\eps^{2\be}\int_{0}^te^{2\lambda s}\left\|\FP\left\{\pa_t f_5+v\cdot\na_x f_6\right\}\right\|_2^2ds\\
\leq& C\eps^{2\be}\int_{0}^te^{2\lambda s}e^{-2\la_0s}P^2\left(
\|[u_{1,0},\ta_{1,0}]\|_{H^{14}}+\|[u_{2,0},\ta_{2,0}]\|_{H^{12}}\right)ds\\
\leq& C\eps^{2\be}P^2\left(
\|[u_{1,0},\ta_{1,0}]\|_{H^{14}}+\|[u_{2,0},\ta_{2,0}]\|_{H^{12}}\right),
\end{split}
\end{equation}
and
\begin{equation}\label{es.nopl3}
\begin{split}
\sup\limits_{0\leq s\leq t}&e^{2\lambda s}\left\|\nu^{-1}w_lg (R^{\ell })(s)\right\|^2_{\infty}
\\
\leq& C\sup\limits_{0\leq s\leq t}\|w_l f_1(s)\|^2_{\infty}\sup\limits_{0\leq s\leq t}e^{2\lambda s}\|w_lR^\ell(s)\|_\infty^2
\\&+C\sum\limits_{i=2}^6\eps^{2(i-1)}\sup\limits_{0\leq s\leq t}\|w_l f_i(s)\|^2_{\infty}\sup\limits_{0\leq s\leq t}e^{2\lambda s}\|w_lR^\ell(s)\|_\infty^2\\
&+C\eps^{8-2\be}\sup\limits_{0\leq s\leq t}\|w_lR^\ell\|^2_{\infty}\sup\limits_{0\leq s\leq t}e^{2\lambda s}\|w_lR^\ell(s)\|_\infty^2
\\&+\eps^{2+2\be}\sup\limits_{0\leq s\leq t}e^{2\lambda s}\left\|w_l\left\{\pa_t f_5+v\cdot \na_xf_6\right\}\right\|_\infty^2
+\eps^{4+2\be}\sup\limits_{0\leq s\leq t}e^{2\lambda s}\|w_l\pa_t f_6\|_\infty^2ds\\
\leq& C\left\{\sup\limits_{0\leq s\leq t}\CE(R^\ell)(s)+\vps_0^2+\eps^2P^2\left(
\|[u_{1,0},\ta_{1,0}]\|_{H^{16}}+\|[u_{2,0},\ta_{2,0}]\|_{H^{14}}\right)\right\}\sup\limits_{0\leq s\leq t}\mathcal {E}(R^\ell)(s)
\\&+\eps^{2+2\be}\sup\limits_{0\leq s\leq t}\left\{e^{-2\lambda_0 s}e^{2\la s}P^2\left(
\|[u_{1,0},\ta_{1,0}]\|_{H^{16}}+\|[u_{2,0},\ta_{2,0}]\|_{H^{14}}\right)\right\}.
\end{split}
\end{equation}
To this end, one has from \eqref{pro24}, \eqref{es.nopl1}, \eqref{es.nopl2} and \eqref{es.nopl3} that
\begin{equation}\label{X1}
\begin{split}
\FX_\de(R^{\ell+1})(t)\leq& C\CE(R_0)(0)+\eps^{2\be}\left\{P^2\left(
\|[u_{1,0},\ta_{1,0}]\|^2_{H^{16}}+\|[u_{2,0},\ta_{2,0}]\|^2_{H^{14}}\right)\right\}
\\&+C\left\{\vps_0^2+\eps^2P^2\left(
\|[u_{1,0},\ta_{1,0}]\|^2_{H^{16}}+\|[u_{2,0},\ta_{2,0}]\|^2_{H^{14}}\right)\right\}\sup\limits_{0\leq s\leq t}\FX_\de(R^\ell)(s)\\&+C\FX_\de^2(R^{\ell})(t),
\end{split}
\end{equation}
which further implies $\FX_\de(R^{\ell+1})(t)<\de$ if $R^{\ell}\in\FX_\de$ with $\de, \vps_0, \eps$ and $\CE(R_0)$ being small enough.

In what follows we prove the strong convergence of the iteration sequence $\{R^{\ell}\}_{\ell=0}^{\infty}$ constructed above. To do this,
by taking difference of the equations that $R^{\ell +1}$ and $R^{\ell }$ satisfy, we deduce that%
\begin{eqnarray*}
\begin{array}{rll}
\begin{split}
\eps\partial _{t}[R^{\ell +1}-R^{\ell }]+v\cdot \nabla _{x}[R^{\ell
+1}-R^{\ell }]+\frac{1}{\eps}L[R^{\ell +1}-R^{\ell }] =g(R^\ell)-g(R^{\ell-1}),
\end{split}
\end{array}
\end{eqnarray*}%
with $R^{\ell +1}-R^{\ell }=0$ initially. By the same fashion as for obtaining \eqref{X1}, one
obtains
\begin{equation*}\label{Xfmin}
\begin{split}
\FX_\de(R^{\ell+1}-R^{\ell})(t)\leq& C\left\{\vps_0^2+\eps^2P\left(
\|[u_{1,0},\ta_{1,0}]\|_{H^{16}}+\|[u_{2,0},\ta_{2,0}]\|_{H^{14}}\right)\right\}\FX_\de(R^\ell-R^{\ell-1})(t)
\\&+
C\left\{\FX_\de(R^\ell)+\FX_\de(R^{\ell-1})\right\}\FX_\de(R^{\ell}-R^{\ell-1})(t).
\end{split}
\end{equation*}
Thus $\{R^{\ell}\}_{\ell=0}^{\infty}$ is a Cauchy sequence in $\FX_\de$ for $\de$ suitably small. Moreover, take $R$ as the limit of the sequence $\{R^{\ell}\}_{\ell=0}^{\infty}$ in $\FX_\de$, then $R$ satisfies
\begin{equation*}\label{sol.es}
\begin{split}
\sup\limits_{0\leq s\leq t}\CE(R)(s)+\int_0^t\mathcal {D}(R)(s)ds
\leq& C\CE(R)(0)+C\eps^{2\be}\left\{P^2\left(
\|[u_{1,0},\ta_{1,0}]\|_{H^{16}}+\|[u_{2,0},\ta_{2,0}]\|_{H^{14}}\right)\right\}.
\end{split}
\end{equation*}
The proof for the uniqueness of the solution obtained above is standard, and
the proof of the positivity of $\mu+\eps\sqrt{\mu}\left\{\sum\limits_{i}^6\eps^{i-1}f_i+\eps^{4-\beta}R\right\}$
is the same as that of Section 3.8 in \cite[pp.66]{EGKM-15} and thus will be omitted. This ends the proof of Theorem \ref{mre}.

\end{proof}

\medskip

\noindent {\bf Acknowledgements:} YG was supported in part by NSFC grant 10828103, DMS 1611695 and Simon Research
Fellowship.
SQL was
supported by grants from the National Natural Science Foundation of China (contracts: 11471142, 11271160  and 11571063). SQL would like to thank Professor Guilong Gui for the helpful discussions on subject of the paper.



\begin{thebibliography}{99}
\bibitem{Bardos-Golse-Levermore-1991}C. Bardos, F. Golse and C. D. Levermore, Fluid dynamic limits of the kinetic equation. I. Formal derivation. {\it J. Statist. Phys.} {\bf 63} (1991), no. 1-2, 323--344.

\bibitem{Bardos-Golse-Levermore-1993}C. Bardos, F. Golse and C. D. Levermore, Fluid dynamic limits of the kinetic equation. II. Convergence proofs for the Boltzmann equation. {\it Comm. Pure. Appl. Math.} {\bf 46} (1993), no. 5, 667--753.

\bibitem{Bardos-Ukai}C. Bardos and S. Ukai, The classical incompressible Navier-Stokes limit of the Boltzmann equation. {\it Math. Models Methods Appl. Sci.} {\bf 1} (1991), no. 2, 235--257.

\bibitem{Bardos-Levermore-Ukai-Yang-2008} C. Bardos, C. D. Levermore, S. Ukai, and T. Yang, Kinetic equations: fluid dynamical limits and viscous heating. {\it Bull. Inst. Math. Acad. Sin. (N.S.)} {\bf 3} (2008), 1--49.

\bibitem{BO-2013}
$\acute{A}$, B$\acute{e}$nyi and  T. Oh, The Sobolev inequality on the torus revisited. {\it Publ. Math. Debrecen}  {\bf 83}  (2013),  no. 3, 359--374.

\bibitem{Caflisch-1980}R. E. Caflisch, The fluid dynamic limit of the nonlinear Boltzmann equation. {\it Comm. Pure. Appl. Math.} {\bf 33} (1-2) (1980), no. 5, 651--666.




\bibitem{Chapman-Cowling-1990} S. Chapman  and T. G. Cowling, \textit{ The Mathematical
 Theory of Non-Uniform Gases}. Cambridge University Press, 1990,
 3rd edition.

\bibitem{DeMasi-Esposito-Lebowitz-1989}A. De Masi, R. Esposito, and J. L. Lebowitz, Incompressible Navier-Stokes and Euler limits of the Boltzmann equation. {\it Comm. Pure Appl. Math.} {\bf 42} (1989), no. 8, 1189--1214.

\bibitem{DiPerna-Lions-1989} R. J. DiPerna and P. L. Lions, On the Cauchy problem for Boltzmann equation: global existence and weak stability. {\it Ann. Math.} {\bf 130} (1989), 321--366.

\bibitem{DL-VPB}
R.-J. Duan and S.-Q. Liu, Stability of the rarefaction wave of the Vlasov-Poisson-Boltzmann system. {\it SIAM J. Math. Anal.} {\bf 47} (2015), no. 5, 3585--3647.


\bibitem{Duan-Ukai-Yang-Zhao} R.-J. Duan, S. Ukai, T. Yang, and H.-J. Zhao, Optimal decay estimates on the linearized Boltzmann equation with time dependent force and their applications. {\it Comm. Math. Phys.} {\bf 277} (2008), no. 1, 189--236.

\bibitem{EGKM-13}R. Esposito, Y. Guo, C. Kim and R. Marra,
Non-isothermal boundary in the Boltzmann theory and Fourier law, {\it Comm. Math. Phys.}  {\bf 323}  (2013),  no. 1, 177--239.

\bibitem{EGKM-15}R. Esposito, Y. Guo, C. Kim and R. Marra,
Stationary solutions to the Boltzmann equation in the Hydrodynamic limit. arXiv:1502.05324.

\bibitem{Esposito-Pulvirenti-2004}R. Esposito and M. Pulvirenti, From particle to fluids, in ``Handbook of Mathematical Fluid Dynamics,¡± Vol. III, North-Holland, Amsterdam, (2004), 1--82.

\bibitem{Glassey-1996} R. T. Glassey, {\it The Cauchy Problem in Kinetic Theory}. SIAM, Philadelphia, 1996.

\bibitem{Golse-2005}F. Golse, The Boltzmann equation and its hydrodynamic limits. Evolutionary equations. Vol. II, 159--301, Handb. Differ. Equ., Elsevier/North-Holland, Amsterdam, 2005.


\bibitem{Gosle-Saint-Raymond-2004}F. Golse and L. Saint-Raymond, The Navier-Stokes limit of the Boltzmann equation for bounded collision kernels. {\it Invent. Math}. {\bf 155} (2004), no. 1, 81--161.

\bibitem{Golse-Saint-Raymond-2009} F. Golse and L. Saint-Raymond, The incompressible Navier-Stokes limit of the Boltzmann equation for hard cutoff potentials. {\it J. Math. Pures Appl.} (9) {\bf 91} (2009), no. 5, 508--552.


\bibitem{Grad-1965}H. Grad,
Asymptotic equivalence of the Navier-Stokes and nonlinear Boltzmann equations.   {\it Proc. Sympos. Appl. Math.}, Vol. XVII, Amer. Math. Soc., Providence, (1965), 154--183.





\bibitem{Guo-2006}Y. Guo, Boltzmann diffusive limit beyond the Navier-Stokes approximation. {\it Comm. Pure. Appl. Math.} {\bf 55} (2006), no. 9, 0626--0687.

\bibitem{Guo-2010}
Y. Guo,
Decay and continuity of the Boltzmann equation in bounded domains.
{\it Arch. Ration. Mech. Anal.} {\bf 197} (2010), no. 3, 713--809.

\bibitem{Guo-Jang-2010} Y. Guo and J. Jang, Global Hilbert expansion for the Vlasov-Poisson-Boltzmann system. {\it Comm. Math. Phys.} {\bf 299} (2010), no. 2, 469--501.

\bibitem{Guo-Jang-Jiang-2009} Y. Guo, J. Jang and N. Jiang, Local Hilbert expansion for the Boltzmann equation. {\it Kinet. Relat. Models} {\bf 2} (2009), no. 1, 205--214.

\bibitem{Guo-Jang-Jiang-2010}Y. Guo, J. Jang and N. Jiang, Acoustic limit for the Boltzmann equation in optimal scaling. {\it Comm. Pure Appl. Math.} {\bf 63} (2010), no. 3, 337--361.

\bibitem{Hilbert} D. Hilbert, \textit{Grundz\"uge einer Allgemeinen
 Theorie der Linearen Integralgleichungen}. Teubner, Leipzig, Chap. 22.

\bibitem{Huang-Wang-Yang-2010-1} F.-M. Huang, Y. Wang and T. Yang, Hydrodynamic limit of the Boltzmann equation with contact discontinuities. {\it Comm. Math. Phys.} {\bf 295} (2010), 293--326.

\bibitem{Huang-Wang-Yang-2010-2}F.-M. Huang, Y. Wang and T. Yang, Fluid dynamic limit to the Riemann solutions of Euler equations: I. Superposition of rarefaction waves and contact discontinuity. {\it Kinet. Relat. Models} {\bf 3} (2010), 685--728.

\bibitem{Huang-Wang-Wang-Yang-2013} F.-M. Huang, Y. Wang, Y. Wang and T. Yang, The limit of the Boltzmann equation to the Euler equations for Riemann problems.  {\it SIAM J. Math. Anal.}  {\bf 45}  (2013),  no. 3, 1741--1811.

\bibitem{Huang-Wang-Wang-Yang-2016} F.-M. Huang, Y. Wang, Y. Wang and T. Yang, Justification of limit for the Boltzmann equation related to Korteweg theory. To appear in {\it Quart. Appl. Math.} http://dx.doi.org/10.1090/qam/1440.



\bibitem{Jiang-Levermore-Masmoudi-2010}N. Jiang, C. D. Levermore and N. Masmoudi, Remarks on the acoustic limit for the Boltzmann equation. {\it Comm. Partial Differential Equations} {\bf 35} (2010), no. 9, 1590--1609.

\bibitem{Jiang-Masmoudi-2015}N. Jiang and N. Masmoudi,
Boundary layers and incompressible Navier-Stokes-Fourier limit of the Boltzmann Equation in Bounded Domain (I). arXiv:1510.02977.


\bibitem{Kawashima-Matsumura-Nishida-1979} S. Kawashima, A. Matsumura and  T. Nishida, On the fluid-dynamical approximation to the Boltzmann equation at the level of the Navier-Stokes equation. {\it Comm. Math. Phys.} {\bf 70} (2) (1979), 97--124.

\bibitem{Lachowicz-1987} M. Lachowicz, On the initial layer and the existence theorem for the nonlinear Boltzmann equation. {\it Math. Methods Appl. Sci.} {\bf 9} (1987), no. 3, 342--366.

\bibitem{Lachowicz-1992} M. Lachowicz, Solutions of nonlinear kinetic equations on the level of Navier-Stokes dynamics. {\it J. Math. Kyoto Univ.} {\bf 32} (1992), no. 1, 31-43.

\bibitem{Levermore-Masmoudi-2010} C. D. Levermore and N. Masmoudi, From the Boltzmann equation to an incompressible Navier-Stokes-Fourier system. {\it Arch. Ration. Mech. Anal.} {\bf 196} (2010), 753--809.



\bibitem{Lions-Masmoudi-2001}P. L. Lions and N. Masmoudi, From the Boltzmann equations to the equations of incompressible fluid mechanics. I, II. {\it Arch. Ration. Mech. Anal.} {\bf 158} (2001), no. 3, 173--193, 195--211.

\bibitem{LY-2016}S.-Q. Liu and X.-F. Yang, The initial boundary value problem for the Boltzmann equation with soft potential. To appear in {\it Arch. Ration. Mech. Anal.}   http://dx.doi.org/10.1007/s00205-016-1038-3.

\bibitem{Liu-Yang-Zhao-2014}S.-Q. Liu, T. Yang and H.-J. Zhao, Compressible Navier-Stokes approximation to the Boltzmann equation. {\it J. Differential Equations} {\bf 256} (2014), no. 11, 3770--3816.

\bibitem{Liu-Zhao-2011}S.-Q. Liu and H.-J. Zhao, Diffusive expansion for solutions of the Boltzmann equation in the whole space. {\it J. Differential Equations} {\bf 250} (2011), no. 2 ,  623--674.


\bibitem{Liu-Yang-Yu-Zhao-2006}  T.-P. Liu,  T. Yang, S.-H. Yu and H.-J. Zhao, Nonlinear stability of rarefaction waves for the Boltzmann equation. {\it Arch. Rational Mech. Anal.} {\bf 181} (2) (2006), 333--371.

\bibitem{Liu-Yu-2004} T.-P. Liu and  S.-H. Yu, Boltzmann equation: Micro-macro decompositions  and positivity of shock profiles. {\it Comm. Math. Phys.} {\bf 246} (1) (2004),  133--179.

\bibitem{Masmoudi-Saint-Raymond-2003}N. Masmoudi and L. Saint-Raymond, From the Boltzmann equation to the Stokes-Fourier system in a bounded domain. {\it Comm. Pure Appl. Math.} {\bf 56} (2003), no. 9, 1263--1293.

\bibitem{Matsumura-Nishida-1980} A. Matsumura and T. Nishida, The initial value problem for the equations of motion of viscous and heat-conductive gases. {\it J. Math. Kyoto Univ.} {\bf 26} (1980), 67--104.


\bibitem{Nishida-1978} T. Nishida, Fluid dynamical limit of the nonlinear Boltzmann equation to the level of the compressible Euler equation. {\it Comm. Math. Phys.} {\bf 61} (2) (1978), 119--148.

\bibitem{Saint-Raymond-2009} L. Saint-Raymond, Hydrodynamic limits of the Boltzmann equation. {\it Lecture Notes in Mathematics}, {\bf 1971}. Springer-Verlag, Berlin, 2009.

\bibitem{Temam-1979}
R, Temam, {\it Navier-Stokes equations. Theory and numerical analysis.} Revised edition. With an appendix by F. Thomasset. Studies in Mathematics and its Applications, 2. North-Holland Publishing Co., Amsterdam-New York, 1979. x+519 pp.




\bibitem{Ukai-Asano-1983}S. Ukai and K. Asano, The Euler limit and the initial layer of the nonlinear Boltzmann equation. {\it Hokkaido Math. J.} {\bf 12} (1983), 303--324.


\bibitem{Xin-Zeng-2010} Z.-P. Xin and H.-H. Zeng, Convergence to the rarefaction waves for the nonlinear Boltzmann equation and compressible Navier-Stokes equations. {\it  J. Diff. Eqs.} {\bf 249} (2010), 827--871.


\bibitem{Yang-Zhao-2006}  T. Yang and H.-J. Zhao, A new energy method for the Boltzmann equation. {\it J. Math. Phys.} {\bf 47} (2006), 053301--18.

\bibitem{Yang-Zhao-2005}  T. Yang and H.-J. Zhao, A half-space problem for the Boltzmann equation with specular reflection boundary condition. {\it Comm. Math. Phys.} {\bf 255} (3) (2005), 683--726.

\bibitem{Yu-2005} S.-H. Yu, Hydrodynamic limits with shock waves of the Boltzmann equations. {\it Comm. Pure Appl. Math.} {\bf 58} (2005), 409--443.












\end{thebibliography}
\end{document}